\documentclass[11pt,letterpaper]{amsart}

\usepackage{amscd,amsmath,amssymb,euscript}
\usepackage[frame,cmtip,curve,arrow,matrix,line,graph]{xy}
\usepackage{tikz-cd}
\usepackage{adjustbox}
\usetikzlibrary{matrix}

\usepackage{hyperref}
\usepackage{stmaryrd}

\title[Generators for Hall algebras of quivers with potential]{Generators for K-theoretic Hall algebras of quivers with potential}
\author{Tudor P\u adurariu}
\address{Department of Mathematics, Columbia University, 
2990 Broadway, New York, NY 10027}
\email{tgp2109@columbia.edu}

\date{}

\newtheorem{thm}{Theorem}[section]
\newtheorem{cor}[thm]{Corollary}
\newtheorem{prop}[thm]{Proposition}

\theoremstyle{definition}

\newtheorem{thm*}[thm]{Theorem$^*$}
\newtheorem{remark}[thm]{Remark}

\newcommand{\comment}[1]{}

\renewcommand{\leq}{\leqslant}
\renewcommand{\geq}{\geqslant}

\newcommand{\OO}{\mathcal O}

\newcommand{\WW}{\mathbb{W}}

\newcommand{\X}{\mathcal{X}}

\newcommand{\C}{\mathbb{C}}

\newcommand{\ee}{\underline{e}}
\newcommand{\dd}{\underline{d}}

\begin{document}
\maketitle

\begin{abstract}
K-theoretic Hall algebras (KHAs) of quivers with potential $(Q,W)$ are a generalization of preprojective KHAs of quivers, which are conjecturally positive parts of the Okounkov--Smironov quantum affine algebras. In particular, preprojective KHAs are expected to satisfy a Poincaré-Birkhoff-Witt theorem. We construct semi-orthogonal decompositions of categorical Hall algebras using techniques developed by Halpern-Leistner, Ballard--Favero--Katzarkov, and \v{S}penko--Van den Bergh. For a quotient of $\text{KHA}(Q,W)_{\mathbb{Q}}$,
we refine these decompositions and prove a PBW-type theorem for it. The spaces of generators of $\text{KHA}(Q,0)_{\mathbb{Q}}$ are given by (a version of) intersection K-theory of coarse moduli spaces of representations of $Q$.   
\end{abstract}

\section{Introduction}

\subsection{Hall algebras of quivers with potential}\label{haq} Let $Q$ be a quiver with vertex set $I$, edge set $E$, and let $W$ be a potential of $Q$. For a dimension vector $d\in\mathbb{N}^I$, denote by $\X(d)=R(d)/G(d)$ the moduli stack of representations of $Q$ of dimension $d$. The potential induces a regular function
\begin{equation}\label{regfct}
\text{Tr}\,W:\X(d)\to\mathbb{A}^1_{\mathbb{C}}.
\end{equation}
Kontsevich--Soibelman \cite{ks} defined a Cohomological Hall algebra of $(Q,W)$, inspired by the Donaldson-Thomas theory of the Jacobi algebra $\text{Jac}\,(Q,W)$, on the cohomology of vanishing cycles of \eqref{regfct}: 
$$\text{CoHA}\,(Q, W):=\bigoplus_{d\in \mathbb{N}^I}H^{\cdot}\left(\X(d), \varphi_{\text{Tr}\,W}\mathbb{Q}\right),$$ where the multiplication $m=p_*q^*$ is defined using the maps \begin{equation}\label{pq}
    \X(d)\times \X(e)\xleftarrow{q} \X(d,e)\xrightarrow{p} \X(d+e)
    \end{equation} from the stack $\X(d,e)$ parametrizing pairs of representations $A\subset B$ with $A$ of dimension $d$ and $B$ of dimension $d+e$. In \cite[Section 8.1]{ks}, Kontsevich--Soibelman proposed a categorification of the CoHA using categories of matrix factorizations of
    the regular function $\text{Tr}\,W$:
  $$\text{HA}(Q,W):=\bigoplus_{d\in\mathbb{N}^I} \text{MF}\left(\X(d), W\right).$$  
  The category $\text{MF}\left(\X(d), W\right)$ is equivalent to the category of singularities $D_{\text{sg}}\left(\X(d)_0\right)$ of the zero fiber of $\text{Tr}\,W$.
The category $\text{HA}(Q,W)$ is monoidal with respect to $m$, see \cite[Theorem 1.1]{P1}, and it is called the categorical Hall algebra (HA) of $(Q,W)$. Its $K_0$ is called the K-theoretic Hall algebra (KHA) of $(Q,W)$. For any edge $e\in E$, let $\mathbb{C}^*$ act on $R(d)$ by scaling the linear map corresponding to $e$. Let $\left(\mathbb{C}^*\right)^E$ be the product of these multiplicative groups.
One can consider a graded version of Hall algebras $\text{HA}^{\text{gr}}(Q,W)$ for any group $\mathbb{C}^*\subset \left(\mathbb{C}^*\right)^E$ scaling $\text{Tr}\,W$
with weight $2$.

\subsection{Preprojective Hall algebra}\label{pha}
For a quiver $Q$ and a dimension vector $d\in\mathbb{N}^I$, let $\mathfrak{P}(d)$ be the moduli stack of dimension $d$ representations of the preprojective algebra of $Q$. 
Varagnolo--Vasserot \cite{VV} studied the categorical preprojective Hall algebra of $Q$:
\[\text{HA}(Q):=\bigoplus_{d\in\mathbb{N}^I}D^b\left(\mathfrak{P}(d)\right).\] 
Its $K_0$ is called the preprojective KHA of $Q$. A $\mathbb{C}^*$-equivariant version is conjectured \cite[Conjecture 1.2]{P1} to be the positive part of the Okounkov--Smirnov quantum affine algebra $U_q\left(\widehat{\mathfrak{g}_Q}\right)$ \cite{os}. There is also a preprojective CoHA of $Q$ \cite{YZ}, \cite{SV}, and a $\mathbb{C}^*$-equivariant version is conjectured \cite{d} to be the positive part of the Maulik--Okounkov Yangian $Y_{\text{MO}}\left(\mathfrak{g}_Q\right)$ \cite{mo}.

To any quiver $Q$ one associates a quiver with potential $\left(\widetilde{Q}, \widetilde{W}\right)$ with a natural grading $\mathbb{C}^*\subset \left(\mathbb{C}^*\right)^{\widetilde{E}}$ such that there is an equivalence of underlying categories
\begin{equation}\label{Isik}
\text{HA}^{\text{gr}}\left(\widetilde{Q}, \widetilde{W}\right)\cong \text{HA}(Q),
\end{equation}
see \cite[Subsection 3.2.3]{P1} for a comparison of the multiplications of the two Hall algebras.

\subsection{The PBW theorem for CoHAs}
Davison--Meinhardt \cite[Theorem C]{dm} proved a 
Poincaré-Birkhoff-Witt (PBW) theorem beyond the case of $\left(\widetilde{Q}, \widetilde{W}\right)$, namely for all pairs $(Q,W)$ where $Q$ is symmetric. 
Let $X(d)$ be the coarse space of $\X(d)$. The coarse space map \[\pi: \X(d)\to X(d)\] induces an increasing (perverse) filtration
\[P^{\leq i}\subset H^{\cdot}\left(\X(d),\varphi_{\text{Tr}\,W}\mathbb{Q}\right)\]
starting in degree $1$.
Denote by $\text{gr}^\cdot\, \text{CoHA}\,(Q, W)$ the associated graded of the $\text{CoHA}\,(Q, W)$ with respect to the filtration $P^{\leq i}$.
 The Davison--Meinhardt PBW theorem for CoHAs says there is an isomorphism of algebras:
\begin{equation}\label{PBWDM}
    \text{gr}^\cdot\, \text{CoHA}\,(Q, W)\cong \text{Sym}\left(P^{\leq 1}\otimes H^{\cdot}\left(B\C^*\right)\right),\end{equation}
    where the right hand side is a super-symmetric algebra and
the generator of $H^{\cdot}\left(B\C^*\right)$ is in perverse degree $2$. As a corollary of \eqref{PBWDM}, there is a natural Lie algebra structure (called the BPS Lie algebra) on $P^{\leq 1}$. There is an isomorphism of $\mathbb{N}^I$-graded vector spaces:
\begin{equation}\label{pl1}
    P^{\leq 1}\cong \bigoplus_{d\in\mathbb{N}^{I}} H^{\cdot}\left(X(d), \varphi_{\text{Tr}\,W}IC_{X(d)}\right).\end{equation}
To prove the PBW theorem, Davison--Meinhardt first formulate a version of \eqref{PBWDM} for the analogous algebra of constructible sheaves on the coarse spaces $X(d)$. The map $\pi:\X(d)\to X(d)$ can be approximated by proper maps and thus $\pi_*$ commutes with vanishing cycles. Therefore the sheaf version of \eqref{PBWDM} for a general potential follows formally from the one for $W=0$. The sheaf version of \eqref{PBWDM} for $W=0$ follows from an explicit description of the summands of $R\pi_*IC_{\X(d)}$ due to
Meinhardt--Reineke \cite[Proposition 4.3 and Theorem 4.6]{mr}. 

\subsection{Semi-orthogonal decompositions for categorical HA}\label{subsec14}

We are interested in proving versions of \eqref{PBWDM} for categorical and K-theoretic Hall algebras. The main tools used in proving \eqref{PBWDM}, intersection complexes and the BBDG decomposition theorem \cite{BBDG}, do not admit obvious generalizations  
to the categorical or the K-theoretic setting. 

To prove a version of \eqref{PBWDM} for categorical HAs, we replace the use of the BBDG decomposition theorem by semi-orthogonal decompositions inspired by work of Halpern-Leistner \cite{hl}, Ballard--Favero--Katzarkov \cite{bfk}, and \v{S}penko--Van den Bergh \cite{sp}, \cite{sp2}. 

For $w\in \mathbb{Z}$, denote by $\text{MF}(\X(d), W)_w$ the category of matrix factorizations on which $z\cdot\text{Id}\subset G(d)$ acts with weight $w$.
We denote by $\dd=(d_1,\cdots, d_k)$ a partition of $d$, by $\X(\dd)$ the stack of representations $0=R_0\subset R_1\subset\cdots\subset R_k$ such that $R_i/R_{i-1}$ has dimension $d_i$ for $1\leq i\leq k$, and by $p_{\dd}$ and $q_{\dd}$ the natural maps
\[\times_{i=1}^k \X(d_i)\xleftarrow{q_{\dd}}\X(\dd)\xrightarrow{p_{\dd}}\X(d).\]
Let $M(d)$ be the weight lattice of a maximal torus of $G(d)$, let $M(d)_\mathbb{R}:=M(d)\otimes_\mathbb{Z}\mathbb{R}$, let $\mathfrak{S}_d$ be the Weyl group of $G(d)$, and let $\delta\in M(d)^{\mathfrak{S}_d}_{\mathbb{R}}$.
We define a set $S^d_w(\delta)$ of partitions $(d_i,w_i)_{i=1}^k$ of $(d,w)$ in Subsection \ref{admissible}. We define categories of generators $\mathbb{M}(d; \delta)_w\subset D^b(\X(d))_w$ after we state Theorem \ref{them1}. For an ordered partition $A=(d_i,w_i)_{i=1}^k$ in $S_w^d(\delta)$, let  \begin{equation}\label{mfa}
    \mathbb{M}_A(\delta):=\otimes_{i=1}^k \mathbb{M}(d_i; \delta_{Ai})_{w_i},
    \end{equation}
    where the weights $\delta_{Ai}\in M(d_i)^{\mathfrak{S}_{d_i}}_{\mathbb{R}}$ are determined by $A$ and $\delta$. Denote by $\mathbb{S}_A(\delta):=\text{MF}\left(\mathbb{M}_A(\delta), \oplus_{i=1}^k W_{d_i}\right)$ the category of matrix factorizations with factors in $\mathbb{M}_A(\delta)$.

\begin{thm}\label{them1}
Assume $Q$ is symmetric. Let $\delta\in M(d)^{\mathfrak{S}_d}_{\mathbb{R}}$.
 There is a semi-orthogonal decomposition
\[\text{MF}\left(\X(d), W\right)_w=\Big\langle 
p_{\dd*}q_{\dd}^*\left(\mathbb{S}_A(\delta)\right)\Big\rangle,\]
where the right hand side is after all ordered partitions $A=(d_i,w_i)_{i=1}^k$ in $S_w^d(\delta)$. The functor $p_{\dd*}q_{\dd}^*$ is fully faithful on the categories $\mathbb{S}_A(\delta)$. The order of the categories is as in Subsection \ref{compadm}, see also Subsection \ref{orderM}. There are analogous semi-orthogonal decompositions for $\text{MF}^{\text{gr}}$.
\end{thm}

Similarly to the Davison--Meinhardt approach, the semi-orthogonal decomposition for $W$ arbitrary follows from the semi-orthogonal decomposition for $W=0$. It is important that we formulate a categorical statement as it is not clear how to make the reduction from $W$ arbitrary to $W=0$ directly in $K$-theory. The categories $\mathbb{M}(d)_w$ are examples of noncommutative resolutions of singularities of $X(d)$ constructed by \v{S}penko--Van den Bergh \cite{sp}. In general, they are not equivalent to $D^b(\mathcal{Y})$ for $\mathcal{Y}$ a stack. 

We now explain the definition of $\mathbb{M}(d; \delta)_w$.
Denote by $M$ the weight lattice of $G(d)$, by $M_{\mathbb{R}}:=M\otimes_{\mathbb{Z}}\mathbb{R}$, by $\nu$ the sum of simple weights of $G(d)$, by $\rho$ half the sum of positive weights of $G(d)$, and by $\mathcal{W}$ the multiset of weights of $R(d)$. Consider the region \begin{equation}\label{polyto}
    \mathbb{W}:=\left(\text{sum}_{\beta\in \mathcal{W}} [0,\beta]\right)\oplus \mathbb{C}\nu \subset M_{\mathbb{R}},
    \end{equation}
    where the Minkowski sum is over weights $\beta$ in $\mathcal{W}$. Denote by $\partial \mathbb{W}$ the boundary of $\mathbb{W}$.
The full subcategory $\mathbb{M}(d; \delta)$ of $D^b(\X(d))$ is generated by the vector bundles $\mathcal{O}_{\X(d)}\otimes \Gamma(\chi)$ for $\chi$ a dominant weight of $G(d)$ such that
\begin{equation}\label{polyt}
\chi+\rho+\delta\in \frac{1}{2}\mathbb{W}.\end{equation}
The category $\mathbb{M}(d; \delta)_w$ is the subcategory of $\mathbb{M}(d; \delta)$ of complexes on which $z\cdot\text{Id}$ acts with weight $w$.

\subsection{Filtrations on KHA}
Inspired by the Davison--Meinhardt PBW theorem \eqref{PBWDM}, we search for filtrations on KHA whose associated graded are $q$-deformed super-symmetric algebras. We replace the perverse filtration $P^{\leq \cdot}$ on $H^{\cdot}\left(\X(d), \varphi_{\text{Tr}\,W}\right)$
with the filtration $F^{\leq \cdot}$ on $K_0(\text{MF}(\X(d),W))$ induced by the semi-orthogonal decompositions from Theorem \ref{them1}. These filtrations depend on 
$\delta_d\in M(d)^{\mathfrak{S}_d}_{\mathbb{R}}$. For generators 
$x_{e,v}\in K_0\left(\mathbb{S}(e;\delta_e)_v\right)$ and $x_{f,u}\in K_0\left(\mathbb{S}(f;\delta_f)_u\right)$, we show that 
\begin{equation}\label{qco}
 x_{e,v}\cdot x_{f,u}=\left(x_{f,u}q^{\gamma(f,e)}\right)\cdot\left(x_{e,v}q^{-\delta(f,e)}\right)
\end{equation}
in the associated graded $\text{gr}^{\cdot} K_0(\text{MF}(\X(d),W))$ with respect to the filtration $F^{\leq \cdot}$,
 where the factors $q^{\gamma(f,e)}, q^{-\delta(f,e)}\in K_0^{T(e)\times T(f)}(\text{pt})$ depend only on $e$ and $f$, see Proposition \ref{commu2}, part (c) for a precise categorical analogue.

\subsection{The PBW theorem for KHAs}\label{pbwkhas}
The categories $\mathbb{M}(d; \delta)_w$ may contain complexes generated in smaller dimensions, that is, complexes in the image of $p_{\dd*}q_{\dd}^*$ for $\dd=(d_i)_{i=1}^k$ a partition of $d$ with $k\geq 2$. These partitions are indexed by a set $T^d_w(\delta)$ defined in Subsection \ref{admissible}. 
For a partition $A=(d_i,w_i)_{i=1}^k$ in $T^d_w(\delta)$, we define a coproduct-type map
\[\Delta_A: K_0\big(\text{MF}\left(\mathbb{M}(d; \delta)_w, W_{d}\right)\big)\to K_0\left(\mathbb{S}_A(\delta)\right).\] 
The inclusion $i_d:\X(d)_0\hookrightarrow \X(d)$ induces an algebra morphism
\begin{equation}\label{algmorp}
    i_*: \text{KHA}(Q,W)\to \text{KHA}(Q,0),
    \end{equation}
see \cite[Proposition 3.6]{P1}. Denote its image by $\text{KHA}(Q,W)'$. It is a $\mathbb{N}^I\times\mathbb{Z}$-graded algebra and denote its $(d,w)\in\mathbb{N}^I\times\mathbb{Z}$-graded part
by $\text{KHA}(Q,W)'_{d,w}$. 
For $d\in \mathbb{N}^I$ and $w\in\mathbb{Z}$,
let $P(d;\delta)_w\subset \text{KHA}(Q,W)'_{d,w}$ be the space of primitive-like elements with respect to $\Delta$, see Subsection \ref{pbwth} for a precise definition. 
For $A=(d_i, w_i)_{i=1}^k$ in a set of partitions $U^d_w(\delta)$, define \[P_A(\delta):=\otimes_{i=1}^k P(d_i; \delta_{Ai})_{w_i},\] where the weights $\delta_{Ai}$ are determined by $A$ and $\delta$. Any summand $(d', w')$ of a partition of $(d,w)$ in $S^d_w(\delta)$ has a corresponding Weyl invariant weight $\delta'$. 
Denote by $\widetilde{T^d_w(\delta)}$ the set of two terms partitions in $T^{d'}_{w'}(\delta')$ for summands $(d', w')$ as above.

\begin{thm}\label{them2}
Let $Q$ be a symmetric quiver with potential $W$. Let $d\in\mathbb{N}^I$, $w\in\mathbb{Z}$, and 
    $\delta_d\in M(d)^{\mathfrak{S}_d}_{\mathbb{R}}$. 
    Then $\text{KHA}(Q,W)'_{d,w}$ is generated (as a $\mathbb{Q}$-vector space) by $P_A(\delta)$ for $A\in U^d_w(\delta)$. The relations between elements in $P_A(\delta)$ are generated by 
\begin{equation}\label{relthem2}
    x_{e,v}\cdot x_{f,u}=\left(x_{f,u}q^{\gamma(f,e)}\right)\cdot\left(x_{e,v}q^{-\delta(f,e)}\right)
\end{equation}
for $x_{e,v}\in P(e;\delta_e)_v$ and $x_{f,u}\in P(f; \delta_f)_u$ for a partition $(e,v), (f,u)$ in $\widetilde{T^d_w(\delta)}$, and where the factors $q^{\gamma(f,e)}, q^{-\delta(f,e)}\in K_0^{T(e)\times T(f)}(\text{pt})$ depend only on $e$ and $f$, see Subsection \ref{gammadelta} for their definitions.
\end{thm}

It suffices to prove Theorem \ref{them2} for $W=0$. In this case, 
Theorem \ref{them2} follows from Theorem \ref{them1} and computations using shuffle formulas from \cite[Section 5]{P3}. In particular, there is an isomorphism \[P(d;0)_0\cong IK_0(X(d)).\] The right hand side is the intersection $K$-theory space (with rational coefficients) of $X(d)$ introduced in \cite{P3}, where a definition of intersection $K$-theory was proposed for a more general class of Artin stacks with good moduli spaces. The spaces $P(d;0)_0$ can be viewed as twisted by $w$ versions of $IK_0(X(d))$. There is a Chern character map to the space of generators of the CoHA: \[\text{ch}:\bigoplus_{d\in\mathbb{N}^I}P(d;0)_0\to P^{\leq 1}.\]

Using the equivalence \eqref{Isik}, we obtain versions of Theorems \ref{them1} and \ref{them2} for the categorical preprojective Hall algebra $\text{HA}(Q)$.



\subsection{Past and future work}
Using a theorem of Toda \cite{t} describing moduli stacks of sheaves on surfaces via quivers with potentials, we prove an analogue of Theorem \ref{them1} for categorical Hall algebras of surfaces in \cite{P4}.

There are versions of the KHA for any stability condition of $Q$. Davison--Meinhardt proved the PBW theorem for a generic stability condition on an arbitrary pair $(Q,W)$; when the stability condition is zero, $Q$ has to be symmetric. There are also $T$-equivariant versions of the KHA for a torus $T\subset \left(\mathbb{C}^*\right)^E$ which fixes $W$. It is interesting to prove versions of Theorems \ref{them1} and \ref{them2} in these more general cases. In the case of tripled quivers and certain equivariant actions and for a Hall algebra constructed using nilpotent representations of $Q$, the algebra morphism \eqref{algmorp} is an inclusion \cite[Subsection 2.4.1]{VV}, see also \cite[Remark 4.7]{SV} for the analogous statement in cohomology, and thus an equivariant version of Theorem \ref{them2} would imply a PBW theorem for the full preprojective KHA. 

It is also interesting to study the dependence on $\delta$ of the spaces of generators $\mathbb{M}(d;\delta)_w$ and $P(d; \delta)_w$. We plan to return to these problems in future work.

\subsection{Structure of the paper}
In Section \ref{s2}, we introduce notations and recall the definitions of categories of singularities and matrix factorizations. In Section \ref{s23}, we prove preliminary results about weight spaces of $G(d)$ and partitions of $(d,w)$ for $d\in\mathbb{N}^I$ and $w\in\mathbb{Z}$.
In Section \ref{s3}, we prove Theorem \ref{them1}, we discuss a categorical version of the $q$-deformed commutator, and prove a categorical analogue of \eqref{qco}. In Section \ref{s4}, we construct the coproduct-type map $\Delta$ and prove Theorem \ref{them2}.

\subsection{Acknowledgements} I would like to thank my PhD advisor Davesh Maulik for suggesting the problem discussed in the present paper and for his constant help and encouragement throughout the project. I would like to thank Ben Davison, Pavel Etingof, Daniel Halpern-Leistner, Andrei Negu\c{t}, Andrei Okounkov, Yukinobu Toda, and Eric Vasserot for useful conversations about the project. I thank the referee for many useful comments.



\section{Notations and preliminaries}\label{s2}

\subsection{Table}
We list the most frequent notations used in the paper:

\begin{figure}
	\centering
\scalebox{0.7}{
	\begin{tabular}{|l|l|l|}
		\hline
	Notation & Description & Location defined \\\hline
$\X(d)=R(d)/G(d)$ & moduli of representations of a quiver & Subsection \ref{subsec211}
\\ \hline
$Q^o$ & quiver with one vertex and no loops &
Subsection \ref{subsec211}
\\ \hline
$T(d)\subset G(d)$ & maximal torus & Subsection \ref{subsec212}
\\ \hline
$M, (M_\mathbb{R}, M^+$ etc.) & weight space (real weights, dominant weights etc.) &
Subsection \ref{subsec212}
\\ \hline
$N$ & coweight lattice &
Subsection \ref{subsec212}
\\ \hline
$\beta^i_j$ & weights of the standard representation & 
Subsection \ref{subsec212}
\\ \hline
$\rho_d$ & half the sum of positive roots of $G(d)$ &
Subsection \ref{subsec212}
\\ \hline
$1_d$ & diagonal cocharacter & Subsection \ref{subsec212}
\\ \hline
$\mathfrak{S}_d$ & Weyl group of $G(d)$ & Subsection \ref{subsec212}
\\ \hline
$\nu_d, \tau_d$ & Weyl invariant weights & Subsection \ref{subsec212}
\\ \hline
$\mathcal{W}$ & set of weights associated to $R(d)$ &
Subsection \ref{subsec212}
\\ \hline
$SG(d)$ & determinant one matrices &
Subsection \ref{subsec212}
\\ \hline
$p_\lambda, q_\lambda$ & maps used to define the Hall product & Subsection \ref{nm}
\\ \hline
$n_\lambda$ & width of the magic categories &
Subsection \ref{nm}
\\ \hline
$\mu\succeq \lambda$  & comparison of cocharacters &
Subsection \ref{subseccompa}
\\ \hline
$\ee\geq\dd$ & comparison of partitions &
Subsection \ref{compa}
\\ \hline
$\mathcal{T}$ & tree used to define paths of partitions &
Subsection \ref{tree}
\\ \hline
$\mathbb{W}, \mathbb{W}_c$ & polytopes used to define quasi-BPS categories
& Subsection \ref{polytope}
\\ \hline
$\mathbb{W}^o$ & interior of polytope & Subsection \ref{polytope}
\\ \hline
$F_r(\lambda)$  & Face of the polytope $\textbf{W}(d)$ & Subsection \ref{polytope}
\\ \hline
$\mathbb{M}(d; \delta)$ & Categories of generators &
Subsection \ref{subsec26}
\\ \hline
$\mathbb{S}(d;\delta)$ & categories of generators for arbitrary potential &
Subsection \ref{subsec26}
\\ \hline
$r$-invariant & quantity used to decompose weights  &
Subsection \ref{subsec311}
\\ \hline
$p$-invariant & quantity used to decompose weights  &
Subsection \ref{subsec311}
\\ \hline
 $F_r(\lambda)^{\mathrm{int}}$ & interior of a face of the polytope $\mathbb{W}$ &  Proposition \ref{prop}
\\ \hline
$A_\chi$ & partition associated to $\chi$ & Subsection \ref{admissible}
\\ \hline
$S^d_w(\delta)$ & set of partitions used in SOD &
Subsection \ref{admissible}
\\ \hline
$\chi_A$ & weight associated to a partition $A$ &
Subsection \ref{admissible}
\\ \hline
$T^d_w(\delta)$ & set of partitions
&Subsection \ref{subsec332}
\\ \hline
$U^d_w(\delta)$ & set of partitions
&Subsection \ref{subsec332}
\\ \hline
$O$ & set used in the order of SODs
&Subsection \ref{compadm}
\\ \hline
$\mathbb{M}_A(\delta)$, $\mathbb{S}_A(\delta)$ & Hall products of categories of generators
& Subsection \ref{subsec41}
\\ \hline
$B_{d,w}$  & set of $(r,p)$-invariants
& Subsection \ref{subsec431}
\\ \hline
$D^b(\X)_{\leq A}$ & subcategory of $D^b(\X)$ &
Subsection \ref{subsec432}
\\ \hline
$\mathcal{W}_{e,f}$ & sets of weights
&Subsection \ref{gammadelta}
\\ \hline
$L_{e,f}$, $N_{e,f}$, $\rho_{e,f}$ & weights
&Subsection \ref{gammadelta}
\\ \hline
$q^{\gamma(e,f)}$, $q^{-\delta(e,f)}$ & equivariant monomials 
&Subsection \ref{gammadelta}
\\ \hline
$w_{e,f}$ & element of the Weyl group
& Subsection \ref{gammadelta}
\\ \hline
$\Delta_A, \Delta_{AB}$ & coproduct-like maps
& Subsection \ref{subsec511}
\\ \hline
$K_0(\mathbb{S}_A(\delta))'$ & quotient of $K_0(\mathbb{S}_A(\delta))$
& Subsection \ref{subsec511}
\\ \hline
$\widetilde{m\boxtimes m}$ & twisted Hall product
& Subsection \ref{subsec541}
\\ \hline
$\textbf{S}$ & set of partitions of a dimension vector
& Subsection \ref{subsec541}
\\ \hline
$P(d;\delta)_w$ & space of primitive elements
& Subsection \ref{pbwth}
\\ \hline
$P_A(\delta)$ & space of Hall products of primitive elements
& Subsection \ref{pbwth}
\\ \hline
	\end{tabular}
}
	\vspace{.5cm}
	\caption{Notation introduced in the paper}
	\label{table:notation}
\end{figure}

\subsection{Notations}

All stacks and schemes considered are over $\mathbb{C}$. For a stack $\X$, denote by $\text{Coh}\,(\X)$ the abelian category of sheaves on $\X$, by $D^b(\X)$ the derived category of bounded complexes of coherent sheaves on $\X$, and by $\text{Perf}\,(\X)$ the category of perfect complexes on $\X$. The categories considered are dg and we denote by $\otimes$ the product of dg categories \cite[Subsections 2.2 and 2.3]{K}.

The functors considered in the paper are derived unless otherwise stated. 
For a morphism $f$, denote by $\mathbb{L}_f$ its cotangent complex. For $\X$ a stack, denote by $\mathbb{L}_{\X}$ the cotangent complex of $\X\to\text{Spec}\,\C$.

We denote by $K_0$ the Grothendieck group of a triangulated category. For $\X$ a stack, denote by $K_0(\X):=K_0\left(\text{Perf}(\X)\right)$ and by $G_0(\X):=K_0\left(D^b(\X)\right)$.
For $M$ a $\mathbb{Z}$-module, denote by $M_{\mathbb{Q}}:=M\otimes_{\mathbb{Z}}\mathbb{Q}$.

\subsubsection{}\label{subsec211}
Let $Q=(I,E)$ be a symmetric quiver with source and target maps $s,t:E\to I$. Let $W$ be a potential. For $d\in\mathbb{N}^I$, let \[\X(d):=R(d)/G(d)\] be the stack of representations of $Q$ of dimension $d$. For $d,e\in\mathbb{N}^I$, let \[\chi(d,e):=\sum_{i\in I}d_ie_i-\sum_{a\in E}d_{s(a)}e_{t(a)}.\] The quiver $Q$ is symmetric, so $\chi(d,e)=\chi(e,d)$.
Let \begin{equation}\label{genK}
    K_0^{T(d)}(\text{pt})=\mathbb{Z}\left[q_{ij}^{\pm 1}\Big|\,i\in I, 1\leq j\leq d_i\right].\end{equation}
We denote by $Q^o$ the quiver with one vertex and no loops.

We may use the notation $W_d$ when we want to emphasize the dimension for the regular function \eqref{regfct}, but we usually write $W$. 
We assume that \eqref{regfct} has $0$ as the only critical value. We denote by $\X(d)_0$ the (derived) zero fiber of \eqref{regfct}.

We will use categories of matrix factorizations $\text{MF}$ or categories of graded matrix factorizations $\text{MF}^{\text{gr}}$, where the grading is for any group $\mathbb{C}^*\subset \left(\mathbb{C}^*\right)^E$ scaling $\text{Tr}\,W_d$ with weight $2$ for any $d\in\mathbb{N}^I$.
These categories $\text{MF}^{\text{gr}}$ depend on which $\mathbb{C}^*$ as above we choose, but they have the same $K_0$ as $\text{MF}$ by \cite[Corollary 3.13]{T2}.


\subsubsection{}\label{subsec212} Let $d\in\mathbb{N}^I$.
Fix maximal torus and Borel subgroups $T(d)\subset B(d)\subset G(d)$. We use the convention that the weights of the Lie algebra of $B(d)$ are negative; it determines a dominant chamber of weights of $G(d)$. Denote by $M$ the weight space of $G(d)$, let $M_{\mathbb{R}}:=M\otimes_{\mathbb{Z}}\mathbb{R}$, and let $M^+\subset M$ and $M^+_{\mathbb{R}}\subset M_{\mathbb{R}}$ be the dominant chambers. When we want to emphasize the dimension vector, we write $M(d)$ etc. Denote by $N$ the coweight lattice of $G(d)$ and by $N_{\mathbb{R}}:=N\otimes_{\mathbb{Z}}\mathbb{R}$. Let $\langle\,,\,\rangle$ be the natural pairing between $N_{\mathbb{R}}$ and $M_{\mathbb{R}}$.

Denote by $\beta^i_j$ the weights of the standard representation of $T(d)$ for $i\in I$ and $1\leq j\leq d_i$, where $d=(d^i)_{i\in I}\in \mathbb{N}^I$, and by
$\rho_d$ half the sum of positive roots of $G(d)$. We denote by $1_d:=z\cdot\text{Id}$ the diagonal cocharacter of $G(d)$. 
Consider the real weights 
\begin{align*}
    \nu_d&:=\sum_{\substack{i\in I\\ j\leq d_i}} \beta^i_j,\\
    \tau_d&:=\frac{\nu_d}{\langle 1_d, \nu_d\rangle}.
\end{align*}
Denote by $\mathcal{W}$ the multiset of weights of $R(d)$ (counted with multiplicities).
Let $\mathfrak{S}_d$ be the Weyl group of $G(d)$ and let $\chi$ be a weight. Denote by $w\chi$ the standard Weyl action and by $w*\psi:=w(\psi+\rho)-\rho$ the shifted Weyl group action.
For a weight $\chi$ in $M$, let $\chi^+$ be the dominant weight in the shifted Weyl orbit of $\chi$ if there is such a weight, or zero otherwise.
For $\chi\in M(d)^+$, let $\Gamma_{G(d)}(\chi)$ be the representation of $G(d)$ of highest weight $\chi$; when there is no danger of confusion, we drop the subscript.

Denote by $SG(d):=\text{ker}\left(\text{det}: G(d)\to\mathbb{C}^*\right)$.

\subsection{Cocharacters and the multiplication map}

\subsubsection{}\label{weight}
The action of $1_d$ on $R(d)$ is trivial. Let $D^b(\X(d))_w$ be the category of complexes on which $1_d$ acts with weight $w$. We have an orthogonal decomposition \[D^b(\X(d))\cong\bigoplus_{w\in\mathbb{Z}} D^b(\X(d))_w.\] 

\subsubsection{}\label{nm}
For a cocharacter $\lambda:\C^*\to SG(d)$, consider the maps of fixed and attracting loci
\begin{equation}\label{e}
\X(d)^\lambda \xleftarrow{q_\lambda}\X(d)^{\lambda\geq 0}\xrightarrow{p_\lambda}\X(d).\end{equation}
We say that two cocharacters $\lambda$ and $\lambda'$ are equivalent and write $\lambda\sim\lambda'$ if $\lambda$ and $\lambda'$ have the same fixed and attracting stacks as above. 

For $\lambda$ a cocharacter of $SG(d)$, we introduce the weights \begin{align*}
L^{\lambda>0}&:=\det \mathbb{L}_{\X(d)}^{\lambda>0}\big|_{\X(d)^\lambda},\\  
N^{\lambda>0}&:=\det \mathbb{L}_{R(d)}^{\lambda>0}\big|_{R(d)^\lambda},
\end{align*} and their analogues for $\lambda\geq 0$, $\lambda<0$, and $\lambda\leq 0$.

For a cocharacter $\lambda$ of $SG(d)$, let $n_\lambda=\big\langle \lambda, L^{\lambda> 0}\big\rangle.$

\subsubsection{}\label{paco} Let $d\in \mathbb{N}^I$. We call $\dd:=(d_i)_{i=1}^k$ a partition of $d$ if $d_i\in\mathbb{N}^I$ are all non-zero and $\sum_{i=1}^k d_i=d$. We define similarly partitions of $(d,w)\in\mathbb{N}^I\times\mathbb{Z}$.
For a cocharacter $\lambda$ of $SG(d)$, there is an associated partition $(d_i)_{i=1}^k$ such that
\[\X(d)^{\lambda\geq 0}\cong \X(\dd)\xrightarrow{q} \X(d)^\lambda\cong\times_{i=1}^k\X(d_i),\]
where recall the stack $\X(\underline{d})$ from Subsection \ref{subsec14}.
 Define the length $\ell(\lambda):=k$. 

Equivalence classes of antidominant cocharacters are in bijection with ordered partitions $(d_i)_{i=1}^k$ of $d$.
For an ordered partition $(d_i)_{i=1}^k$ of $d$, fix a corresponding antidominant cocharacter $\lambda_{\dd}$ of $SG(d)$ which induces the maps
\[\X(d)^\lambda\cong\times_{i=1}^k\X(d_i)
\xleftarrow{q_\lambda}\X(d)^{\lambda\geq 0}\cong \X(\dd)
\xrightarrow{p_\lambda}\X.\] 
The multiplication in the Hall algebra is induced by the functor $p_{\lambda*}q_\lambda^*$. We may drop the subscript $\lambda$ in the functors $p_*$ and $q^*$ when the cocharacter $\lambda$ is clear.

\subsubsection{}\label{id} Let $(d_i)_{i=1}^k$ be a partition of $d$. There is an identification \[\bigoplus_{i=1}^k M(d_i)\cong M(d),\] where the simple roots $\beta^i_j$ in $M(d_1)$ correspond to the first $d_1$ simple roots $\beta^i_j$ of $d$ etc.

\subsubsection{}\label{subseccompa} For cocharacters $\lambda,\mu:\C^*\to SG(d)$, we write $\mu\geq \lambda$ 
if for every weight $\beta$ in $M$ with $\langle \lambda, \beta \rangle> 0$, we have that $\langle \mu, \beta\rangle > 0$. 

\subsection{Partitions}

\subsubsection{}\label{compa}
Let $\underline{e}=(e_i)_{i=1}^l$ and $\dd=(d_i)_{i=1}^k$ be two partitions of $d\in \mathbb{N}^I$. We write $\ee\geq\dd$
if there exist integers \[a_0=0< a_1<\cdots<a_{k-1}\leq a_k=l\] such that for any $0\leq j\leq k-1$, we have
\[\sum_{a_{j}+1}^{a_{j+1}} e_i=d_{j+1}.\]
There is a similarly defined order on pairs $(d,w)\in\mathbb{N}^I\times\mathbb{Z}$.

If $\lambda$ and $\mu$ are cocharacters with associated partitions $\dd$ and $\underline{e}$ such that $\mu\geq \lambda$, then $\ee\geq\dd$.



\subsubsection{}\label{tree}

We define 
$\mathcal{T}$ to be the unique (oriented) tree such that: 
\begin{enumerate}
\item each vertex is indexed by a partition $(d_1, \ldots, d_k)$ of 
	some $d \in \mathbb{N}^I$, 
 \item for each $d \in \mathbb{N}^I$, there is a unique vertex 
	indexed by the partition $(d)$ of size one, 
  \item if $\bullet$ is a vertex indexed by $(d_1, \ldots, d_k)$ 
	and $d_m=(e_1, \ldots, e_s)$ is a partition of $d_m$ for some $1\leq m \leq k$, then there is a unique vertex $\bullet'$ indexed by 
	$(d_1, \ldots, d_{m-1}, e_1, \ldots, e_s, d_{m+1}, \ldots, d_k)$ and
	with an edge from $\bullet$ to $\bullet'$, and 
 \item all edges in $\mathcal{T}$ are as in (3).
	\end{enumerate}
Note that  
 each partition $(d_1, \ldots, d_k)$ of 
	some $d \in \mathbb{N}^I$ 
 gives an index of some (not necessary unique) vertex.  
A subtree $T \subset \mathcal{T}$ is called a \textit{path of partitions} 
if it is connected, contains a vertex indexed by $(d)$ for some 
$d \in \mathbb{N}^I$ and 
a unique end vertex $\bullet$. 
The partition $(d_1, \ldots, d_k)$ at the end vertex $\bullet$
is called the associated partition of $T$. 
We define the Levi group associated to $T$ to be 
\begin{align*}
	L(T):=\times_{i=1}^k G(d_i). 
	\end{align*}
 Note that each edge $\ell\in \mathcal{T}$ corresponds to a partition of some $\mathbb{N}^I$: if $\ell$ connects $(d_1, \ldots, d_k)$ and $(d_1, \ldots, d_{m-1}, e_1, \ldots, e_s, d_{m+1}, \ldots, d_k)$ as in (3), then $\ell$ corresponds to the partition $(e_1,\ldots, e_s)$ of $d_m$ and thus there is an associated cocharacter $\lambda_\ell$ as in Subsection \ref{paco}.

\subsection{Regions in weight spaces}
\subsubsection{}\label{polytope}  
The polytope $\mathbb{W}$ is defined as follows:
\[\mathbb{W}:=\left(\text{sum}_{\beta\in \mathcal{W}} [0,\beta]\right)\oplus \mathbb{R}\tau_d \subset M_{\mathbb{R}},
\]
    where the Minkowski sum is over weights $\beta\in\mathcal{W}$.
Denote by $\mathbb{W}_c\subset \mathbb{W}$ the hyperplane where the coefficient of $\tau_d$ is $c$. Note that 
\[\mathbb{W}=\mathbb{W}_0\oplus \mathbb{R}\tau_d.\]
Let $\mathbb{W}^o:=\mathbb{W}\setminus\partial \mathbb{W}.$

Let $\lambda$ be a cocharacter of $SG(d)$ and let $r\geq 0$. Define the hyperplane  $H_r(\lambda)$ as the locus of weights $\psi\in M_{\mathbb{R}}$ such that
\[\langle \lambda, \psi\rangle+
\langle \lambda, rN^{\lambda>0}\rangle=0.\]
The boundary of the polytope $r\mathbb{W}$ is contained in the subspaces $H_r(\lambda)$ for $\lambda$ a cocharacter of $SG(d)$. Let $F_r(\lambda):=H_r(\lambda)\cap r\mathbb{W}$. When $r=\frac{1}{2}$, we drop the subscript and write $F(\lambda)$.


\subsubsection{} Consider a partition $(d_i)_{i=1}^k$ of $d\in\mathbb{N}^I$ with associated Levi group $L=\times_{i=1}^k G(d_i)$.
Using the identification in Subsection \ref{id}, define the region
\[\mathbb{W}(L):
=\left(\bigoplus_{1\leq i\leq k}\,\mathbb{W}(d_i)_0\right)\oplus \mathbb{R}\tau_d\subset \bigoplus_{1\leq i\leq k}\,M(d_i)_{\mathbb{R}}\cong M(d)_{\mathbb{R}}.\]  

\subsection{Categories of singularities and matrix factorizations}

\subsubsection{} 
Let $Y$ be a (quasi-smooth) scheme with an action of a reductive group $G$. Consider the quotient stack
$\mathcal{Y}=Y/G$.
The category of singularities of $\mathcal{Y}$ is a triangulated category defined as the quotient of triangulated categories
$$D_{\text{sg}}(\mathcal{Y}):=D^b(\mathcal{Y})/\text{Perf}(\mathcal{Y}).$$
If $\mathcal{Y}$ is smooth, then $D_{\text{sg}}(\mathcal{Y})$ is trivial. There is an exact sequence
\begin{equation}\label{es}
    K_0(\mathcal{Y})\to G_0(\mathcal{Y})\to K_0\left(D_{\text{sg}}(\mathcal{Y})\right)\to 0.\end{equation}

\subsubsection{} A reference for the next two subsections is \cite[Section 2.2]{T3}.
Let $\X=X/G$ be a quotient stack with $X$ a smooth affine scheme and consider
a regular function 
\[f:\mathcal{X}\to\mathbb{A}^1_{\mathbb{C}}.\]
Consider the category of matrix factorizations $\text{MF}(\X, f)$. It has objects $(\mathbb{Z}/2\mathbb{Z})\times G$-equivariant factorizations $(P, d_P)$, where $P$ is a $G$-equivariant coherent sheaf, $\langle 1\rangle$ is the twist corresponding to a non-trivial $\mathbb{Z}/2\mathbb{Z}$-character on $X$, and \[d_P: P\to P\langle 1\rangle\] with $d_P\circ d_P=f$. Alternatively, the objects of $\text{MF}(\X, f)$ are tuplets
\[\left(\alpha: F_1\rightleftarrows F_2: \beta\right),\]
where $F_1$ and $G$ are $F_2$-equivariant coherent sheaves, and $\alpha$ and $\beta$ are $G$-equivariant morphisms
such that $\alpha\circ\beta$ and $\beta\circ\alpha$ are multiplication by $f$. Let $\X_0:=0\times_{\mathbb{A}^1_{\mathbb{C}}}\X$. The stack $\X_0$ is quasi-smooth. 
By a theorem of Orlov \cite{o2},  there is an equivalence 
\[D_{\text{sg}}(\X_0)\cong \text{MF}(\X, f).\]
For a triangulated subcategory $\mathcal{A}$ of $D^b(\X)$, define $\text{MF}(\mathcal{A}, f)$ as the full subcategory of $\text{MF}(\mathcal{X}, f)$ with objects pairs $(P, d_P)$ with $P$ in $\mathcal{A}$. 

\subsubsection{}
Assume there exists an extra action of $\mathbb{C}^*$ on $X$ which commutes with the action of $G$ on $X$ and which scales $f$ with weight $2$. Denote by $(1)$ the twist by the character \[\text{pr}_2:G\times\mathbb{C}^*\to\mathbb{C}^*.\]
Consider the category of graded matrix factorizations $\text{MF}^{\text{gr}}(\mathcal{X}, f)$. It has objects pairs $(P, d_P)$ with $P$ an equivariant $G\times\mathbb{C}^*$-sheaf on $X$ and $d_P:P\to P(1)$ a $G\times\mathbb{C}^*$-equivariant morphism. For a triangulated subcategory $\mathcal{B}$ of $D^b_{\mathbb{C}^*}(\X)$, define $\text{MF}^{\text{gr}}(\mathcal{B}, f)$ as the full subcategory of $\text{MF}^{\text{gr}}(\mathcal{X}, f)$ with objects pairs $(P, d_P)$ with $P$ in $\mathcal{B}$. The gradings used in this paper are induced by groups $\mathbb{C}^*$ as described in Subsection \ref{haq}.
For $f$ zero and the trivial $\mathbb{C}^*$-action on $\X$, there is an equivalence
\[\text{MF}^{\text{gr}}(\X,0)\cong D^b(\X),\]
see \cite[Remark 2.3.7]{T3}.

\subsection{Categories of generators}\label{subsec26}

Let $\delta\in M_{\mathbb{R}}^{\mathfrak{S}_d}$.
The category \[\mathbb{M}(d; \delta)\] is generated by the vector bundles $\mathcal{O}_{\X(d)}\otimes \Gamma_{G(d)}(\chi)$ for $\chi$ a dominant weight of $G(d)$ such that
\[\chi+\rho+\delta\in\frac{1}{2}\mathbb{W}.\]
The category $\mathbb{M}(d; \delta)_w\subset \mathbb{M}(d;\delta)$ is the subcategory of complexes of weight $w$ with respect to the diagonal cocharacter. Then there is an orthogonal decomposition:
\[\mathbb{M}(d; \delta)\cong \bigoplus_{w\in\mathbb{Z}}\mathbb{M}(d; \delta)_w\]
The category $\mathbb{M}(d; \delta)$ can be also described as the subcategory of $D^b(\X(d))$ of complexes $\mathcal{F}$ such that 
\[\langle \lambda, \mathcal{F}|_0\rangle
\subset 
\left[-\frac{n_\lambda}{2}-\langle \lambda, \delta\rangle, \frac{n_\lambda}{2}-\langle \lambda, \delta\rangle\right],\]
for all cocharacters $\lambda$ of $SG(d)$, where, by abuse of notation, $\langle \lambda, \mathcal{F}|_0\rangle$ is the set of weights of $\lambda$ on 
$\mathcal{F}|_0$, the restriction of $\mathcal{F}$ to the origin $0$ in $R(d)$, see \cite[Section 2]{hls}.

Let $W$ be a potential for $Q$. Denote by $\mathbb{S}(d; \delta)_w:=\text{MF}\left(\mathbb{M}(d; \delta)_w, W\right)$. For a given grading, denote by $\mathbb{S}^{\text{gr}}(d; \delta)_w:=\text{MF}^{\text{gr}}\left(\mathbb{M}(d; \delta)_w, W\right)$.

\subsection{Semi-orthogonal decomposition}\label{sod4}
\subsubsection{}\label{orderM}
Let $I$ be a set. Assume there is a set $O\subset I\times I$ such that for any $i, j\in I$ we have that $(i,j)\in O$, or $(j,i)\in O$, or both $(i,j)\in O$ and $(j,i)\in O$. An element $o\in I$ is minimal if $I\times o\subset O$. An element $m\in I$ is maximal if $m\times I\subset O$.

Let $\mathbb{T}$ be a triangulated category. We will construct semi-orthogonal decompositions
\[\mathbb{T}=\langle \mathbb{A}_i\rangle,\] 
by subcategories $\mathbb{A}_i$ with $i\in I$
such that for any $i,j\in I$ with $(i, j)\in O$ and objects $\mathcal{A}_i\in\mathbb{A}_i$, $\mathcal{A}_j\in\mathbb{A}_j$, we have that:
\[\text{Ext}^a_{\mathbb{T}}(\mathcal{A}_j,\mathcal{A}_i)=0\] for any $a\in\mathbb{Z}$. 

If $o$ is a minimal element in $I$, then the inclusion $\mathbb{A}_o\hookrightarrow\mathbb{T}$ admits a right adjoint $\mathbb{T}\to \mathbb{A}_o$. If $m$ is a maximal element in $I$, then the inclusion $\mathbb{A}_m\hookrightarrow\mathbb{T}$ admits a left adjoint $R:\mathbb{T}\to \mathbb{A}_m$.

\section{Admissible sets}\label{s23}

For $Q$ a quiver and $d\in\mathbb{N}^I$, denote by $Q^d$ the subquiver of $Q$ with vertices $i\in I$ with $d_i\neq 0$. Recall that $Q^o$ is the quiver with one vertex and no loops.

\subsection{The $r$-invariant}
\subsubsection{}\label{subsec311}
Assume that $Q=Q^d$ is connected and that $Q$ is not $Q^o$.
Denote by $\mathbb{R}^{\leq 0}:=(-\infty, 0]\subset \mathbb{R}$.
Then 
\[\mathbb{R}\tau_d\oplus\bigoplus_{\beta\in\mathcal{W}}\mathbb{R}^{\leq 0}\beta=M_{\mathbb{R}}.\]
Indeed, $Q$ is symmetric, so it suffices to show that all the weights $\beta^i_a-\beta^j_b$ are in the $\mathbb{R}$-span of $\beta\in \mathcal{W}$ for $i, j\in I$, $1\leq a\leq d_i$, $1\leq b\leq d_j$. This is true as there is a path between $i$ and $j$.

For $\chi$ a weight in $M_{\mathbb{R}}$, define its $r$-invariant to be the smallest nonnegative real number $r$ such that
$\chi\in r\WW.$ 
For $\chi$ a real weight in $M_\mathbb{R}$ such that $\chi\in r\partial\WW$, we have that:
\[r=\text{max}\,\frac{\langle \lambda, \chi\rangle}{\langle \lambda, N^{\lambda>0}\rangle}
=-\text{min}\,\frac{\langle \lambda, \chi\rangle}{\langle \lambda, N^{\lambda>0}\rangle},\]
where the min and max are taken over all $SG(d)$ cocharacters $\lambda$, see \cite[Lemma 2.9]{hls} applied to the $T(d)$-representation $R(d)$, for $\varepsilon=0$, and for the (slighly) larger polytope $\mathbb{W}$ considered in this paper. When $\chi\in M^+_{\mathbb{R}}$, we have that:
\[r=-\text{min}\,\frac{\langle \lambda, \chi\rangle}{\langle \lambda, N^{\lambda>0}\rangle},\] where the min is taken over all $SG(d)$ antidominant cocharacters $\lambda$.

Recall that $\mathcal{W}$ is the set of weights of $R(d)$ counted with multiplicity. 
If $\chi\in M_\mathbb{R}$ has $r(\chi)=r$, then there are coefficents $-r\leq c_{\beta}\leq 0$ for $\beta\in\mathcal{W}$ and $c\in\mathbb{R}$ such that
\begin{equation}\label{chibeta}
    \chi=\sum_{\beta\in\mathcal{W}} c_{\beta}\beta+c\tau_d.
    \end{equation}
The $p$-invariant for $\chi$ with $r(\chi)=r$ is defined as the smallest number of coefficients $c_\beta$ for $\beta\in\mathcal{W}$ equal to $r$ in a formula \eqref{chibeta}.

\begin{prop}\label{rem1} (a) Two weights of dimension $d$ that differ by a multiple of $\tau_d$ have the same $(r, p)$-invariant.

(b) Let $\chi$ be a weight, $w\in \mathfrak{S}_d$, and $\delta\in M^{\mathfrak{S}_d}_{\mathbb{R}}$. Then \[(r,p)(\chi+\rho+\delta)=(r,p)(w*\chi+\rho+\delta).\]
\end{prop}

\begin{proof}
The proof of $(a)$ is clear. To prove $(b)$, observe that $\WW\subset M_{\mathbb{R}}$ is Weyl invariant and that $w*\chi+\rho+\delta=w(\chi+\rho)-\rho+\rho+\delta=w(\chi+\rho+\delta).$
\end{proof}

The following is a stronger form of \cite[Lemma 3.12]{hls}.

\begin{prop}\label{prop}
Let $\chi\in M_{\mathbb{R}}$ be a real weight with $r(\chi)=r$ and let $\lambda$ be a cocharacter of $SG(d)$ with $\chi\in F_r(\lambda).$
Then there are coefficients $c\in\mathbb{R}$ and 
\[  c_{\beta}= 
     \begin{cases}
       0 \text{ if }\langle \lambda, \beta\rangle<0,\\
       -r \text{ if } \langle \lambda, \beta\rangle>0, \\
       \text{in }[-r,0] \text{ if }
\langle \lambda, \beta\rangle=0  \\ 
     \end{cases}
\] such that
\begin{equation}\label{chibeta2}
    \chi=\sum_{\beta\in\mathcal{W}} c_{\beta}\beta+ c\tau_d.
    \end{equation}
    Conversely, every $\chi\in r\mathbb{W}$ of the form \eqref{chibeta2} is on $F_r(\lambda)$. Thus if $\mu$ is a cocharacter with $F_r(\mu)\subset F_r(\lambda)$, then $\mu\geq \lambda$. Define
    \[F_r(\lambda)^{\text{int}}:=F_r(\lambda)\setminus\bigcup_{\mu>\lambda}F_r(\mu).\]
If
$\chi\in F_r(\lambda)^{\text{int}},$ then we can choose the coefficients $c_\beta\in (-r, 0]$ in \eqref{chibeta2} for $\langle \lambda, \beta\rangle=0$.
\end{prop}

\begin{proof}
Let $\alpha\in M_{\mathbb{R}}$ be a small weight with $\langle \lambda, \alpha\rangle>0$. Then $\chi-\alpha$ will not be in $r\mathbb{W}$ and thus it will have $r(\chi-\alpha)>r(\chi)$. Further, $\chi+\alpha$ will still be inside $r\mathbb{W}$, and thus $r(\chi+\alpha)\leq r(\chi)$. In the above expression
$$\chi=\sum_{\beta\in\mathcal{W}} c_{\beta}\beta+ c\tau_d,$$
assume that there exists a weight $\langle \lambda, \beta\rangle>0$ such that its coefficients $c_{\beta}>-r$. But then for small $\varepsilon>0$, we will have that $$\chi-\varepsilon \beta=\sum_{\beta\in\mathcal{W}} c'_{\beta} \beta+c\tau_d$$ with all coefficients $-r\leq c'_{\beta}\leq 0$, so $\chi+\varepsilon \beta\in r\mathbb{W}$, which contradicts the above observation. Thus $c_{\beta}=-r$ for all $\beta\in \mathcal{W}$ with $\langle \lambda, \beta\rangle>0$.
The argument for why $c_{\beta}=0$ when $\langle \lambda, \beta\rangle<0$ is similar. The converse statement is clear.  If $F_r(\mu)\subset F_r(\lambda)$, then $\{\beta|\,\langle \lambda, \beta\rangle>0\}\subset\{\beta|\,\langle \mu, \beta\rangle>0\}$, so
$\mu\geq \lambda$.

Assume next that $\chi\in F_r(\lambda)^{\text{int}}.$
We show that one can choose the coefficients $0\geq c_{\beta}>-r$ for weights $\beta$ such that $\langle \lambda, \beta\rangle=0$.
Let $\psi:=\chi+rN^{\lambda>0}$. Then
$$\psi=\sum_{\beta\in\mathcal{W}^\lambda} c_{\beta} \beta+c\tau_d,$$ where $\mathcal{W}^\lambda$ is the multiset $\{\beta| \langle \lambda, \beta\rangle=0\}$ and $-r\leq c_{\beta}\leq 0$.
If $r(\psi)=r$, then by the above argument there exists an antidominant cocharacter $\mu$ such that
$$\psi=\sum_{\beta\in\mathcal{W}^\lambda} c_{\beta} \beta+c\tau_d,$$ and the coefficients are
\[  c_{\beta}= 
     \begin{cases}
       0 \text{ if }\langle \mu, \beta\rangle<0,\\
       -r \text{ if } \langle \mu, \beta\rangle>0, \\
       \text{in }[-r, 0]\text{ if }
\langle \mu, \beta\rangle=0.  \\ 
     \end{cases}
\]
Say that $\lambda$ has associated Levi group $L(\lambda)\subset G(d)$ and that $\mu$ has associated Levi group $L(\mu)\subset L(\lambda)$. Let $\mu'$ be a cocharacter of $SG(d)$ with associated Levi $L(\mu)\subset G(d)$. 
We then have that
$$\langle\mu',\chi\rangle+r\langle\mu',N^{\mu'>0}\rangle=0.$$ 
This means that $\chi\in F_r(\mu')$. We have $\mu'\geq \lambda$ and $\chi\in F_r(\lambda)^{\text{int}}$, so $\lambda=\mu'$, which implies that $r(\psi)<r$. We can thus choose all the coefficients $c_{\beta}$ to be $-r<c_{\beta}\leq 0$ for weights $\beta$ with $\langle \lambda, \beta\rangle=0$.
\end{proof}


\begin{prop}\label{subse}
Let $\chi$ be a weight in $M_{\mathbb{R}}$ with $r(\chi)=r\geq\frac{1}{2}$. Assume that there is a set $J\subset\mathcal{W}$ and $\psi$ a sum of weights in $\mathcal{W}\setminus J$ with positive coefficients $<r$ such that
\begin{equation}\label{chisum}
    \chi=-r\sum_{J}\beta-\psi.
\end{equation}
Assume that $\chi\in F_r(\lambda)^{\text{int}}$. Then $\{\beta|\,\langle \lambda, \beta\rangle>0\}\subset J$.
\end{prop}

\begin{proof}
Denote by $\alpha_+$ a sum of weights in $\{\beta|\,\langle \lambda, \beta\rangle>0\}$ with positive coefficients, by $\alpha_0$ a sum of weights in $\{\beta|\,\langle \lambda, \beta\rangle=0\}$ with positive coefficients, and by $\alpha_-$ a sum of weights in $\{\beta|\,\langle \lambda, \beta\rangle<0\}$ with positive coefficients.
Using Proposition \ref{prop}, there is a weight $\phi_0$ with $r(\phi_0)<r$ such that
\[\chi=-rN^{\lambda>0}+\phi_0.\]
Let $J_+=J\cap\{\beta|\,\langle \lambda, \beta\rangle>0\}$ etc.
Choose $\psi_+$, $\psi_0$, $\psi_-$ such that
\[\chi=-r\sum_{J}\beta-\psi_+-\psi_0-\psi_-.\]
We can thus write
\[rN^{\lambda>0}-r\sum_{J_+}\beta-\psi_+=
r\sum_{J_0}\beta+r\sum_{J_-}\beta+\psi_0-\phi_0+\psi_-.\]
The right hand side has $\lambda$-weight $\leq 0$, while the left hand side has $\lambda$-weight $\geq 0$. The $\lambda$-weight of the left hand side is thus zero and this implies that $J_+=\{\beta|\,\langle \lambda, \beta\rangle>0\}$.
\end{proof}

\subsubsection{}\label{godown2}

The following is an immediate corollary of Proposition \ref{prop}:

\begin{cor}\label{godown}
Let $\chi\in M^+_\mathbb{R}$ with $r(\chi)=r$
and let $\lambda$ be the antidominant cocharacter of $SG(d)$ with Levi group $L$ such that $\chi\in F_r(\lambda)^{\text{int}}$. Then
$$\chi=-rN^{\lambda>0}+\psi,$$ where $r(\psi)=s<r$, $\psi$ is $L$-dominant, and $\psi\in s\mathbb{W}(L)$. 
\end{cor}

Recall the tree $\mathcal{T}$ from Subsection \ref{tree}.
Applying Corollary \ref{godown} repeatedly, we obtain a decomposition of $\chi$, see Proposition \ref{prop:decompchi}. Before we state it, we introduce some notation.

For $d_a$ a summand of a partition of $d$, denote by $M(d_a)\subset M(d)$ the subspace as in the decomposition from Subsection~\ref{id} and let $A\subset \{1,\ldots, d\}$ be the set of indices of weights of standard representation corresponding to $M(d_a)\subset M(d)$. 
Assume that $j$ is a partition of a dimension $d_a\in \mathbb{N}$, alternatively, $j$ is an edge of the tree $\mathcal{T}$ introduced in Subsection \ref{tree}.
Let $\lambda_{j}$ be the corresponding 
antidominant cocharacter of $T(d_a)$, see Subsection \ref{tree}. 
Let $\mathcal{W}_{j}\subset \mathcal{W}$ be the multiset of weights
with $\langle \lambda_{j}, \beta\rangle>0$. 
Define \[N_j:=\sum_{\beta\in\mathcal{W}_{j}}\beta.\]
If $R(d)=\mathfrak{g}(d)$, we use the notation 
\begin{equation}\label{notationmathfrak}
    \mathfrak{g}_j:=N_j.
\end{equation}

\begin{prop}\label{prop:decompchi}
Let $\chi$ be a weight in $M(d)^+_\mathbb{R}$. 
There exists:
\begin{enumerate}
    \item a path of partitions $T$, see Subsection \ref{tree}, 
\item coefficients $r_j$ for $j\in T$ such that $r_{j}>1/2$ 
if $j$ corresponds to a partition with length $>1$, 
and $r_{j}=0$ otherwise; further, if $j, j'\in T$ are vertices 
corresponding to partitions with length $>1$, and 
with a path from $j$ to $j'$, then $r_{j}> r_{j'}> \frac{1}{2}$, and
\item weight $\psi\in \frac{1}{2}\mathbf{W}(d)$
such that:
\end{enumerate}
\begin{equation}\label{godown5}
    \chi=-\sum_{j\in T}r_j N_\ell+\psi.
\end{equation}
\end{prop}

We call the expansion \eqref{godown5} \textit{the standard form of} $\chi\in M^+_\mathbb{R}$.

For future reference, we note the following:

\begin{prop}\label{propnj}
    In the setting of Proposition \ref{prop:decompchi}, let $L$ be the Levi group corresponding to the path of partitions $T$ and let $\mathfrak{S}_L$ be the Weyl group of $L$. Then $\sum_{j\in T}r_jN_j\in M(d)^{\mathfrak{S}_L}_\mathbb{R}$.
\end{prop}

\begin{proof}
    It suffices to show that, in the setting of Corollary \ref{godown}, we have that $N^{\lambda>0}\in M(d)^{\mathfrak{S}_L}_\mathbb{R}$.
    The weight $N^{\lambda>0}$ is a linear combination after the edges of $Q$:
    \[N^{\lambda>0}=\sum_{e\in E}\mathrm{Hom}(\mathbb{C}^{d_{s(e)}}, \mathbb{C}^{d_{t(e)}})^{\lambda>0}\in M(d).\]
    It thus suffices to check the claim when $Q$ is an edge between vertices $i$ and $j$.
    Assume for simplicity that $\lambda$ corresponds to a length $2$ partition $(e,f)$ of $d$. 
    Then 
    \[N^{\lambda>0}=\mathrm{Hom}(\mathbb{C}^{d_{i}}, \mathbb{C}^{d_j})^{\lambda>0}=e_i\left(\sum_{a>e_j}\beta^j_a\right)-f_j\left(\sum_{a\leq e_i}\beta^i_a\right),\] which is indeed invariant under the action of $\mathfrak{S}_{e}\times\mathfrak{S}_f$.
\end{proof}

\subsection{A corollary of the Borel-Weyl-Bott theorem}

For future reference, we state \cite[Section 3.2]{hls}. For $I$ a set of weights, we denote by $\sigma_I$ the sum of the weights in $I$. Let $\chi$ be a weight in $M$. Let $\chi^+$ be the dominant Weyl-shifted conjugate of $\chi$ if it exists, and zero otherwise. Let
$w$ be the element of the Weyl group such that $w*(\chi-\sigma_I)$ is dominant or zero. It has length $\ell(w)=:\ell(I)$.

\begin{prop}\label{bbw}
Let $\lambda$ be a cocharacter of $SG(d)$. Recall the maps $p_{\lambda}$ and $q_{\lambda}$ from \eqref{e}.
Let $\chi\in M^+$. 
Then there is a spectral sequence converging to $p_{\lambda*}q_{\lambda}^*\left(\mathcal{O}_{\X(d)^\lambda}\otimes\Gamma_{G(d)^\lambda}(\chi)\right)$ with terms $\mathcal{O}_{\X(d)}\otimes \Gamma_{G(d)}\left((\chi-\sigma_I)^+\right)$ in degree $|I|-\ell(I)$ corresponding to subsets $I\subset \{\beta|\,\langle \lambda, \beta\rangle<0\}$.
\end{prop}

We will be using Proposition \ref{bbw} in conjunction with the following computation.

\begin{prop}\label{rgoesdown}
Assume that $Q=Q^d$ is connected and that $Q$ is not $Q^o$.
Let $\chi\in M^+_{\mathbb{R}}$ with $r(\chi)=r>\frac{1}{2}$ and let $\lambda$ be the antidominant cocharacter of $SG(d)$ with $\chi\in F_r(\lambda)^{\text{int}}$.
Let
$I$ be a non-empty subset of  $\{\beta|\,\langle \lambda, \beta\rangle<0\}$.
Then $\chi-\sigma_I$ is in $r\mathbb{W}$. If $\chi-\sigma_I$ lies on $F_r(\mu)^{\text{int}}$, then $\mu<\lambda$. In particular, we have that \[(r,p)(\chi-\sigma_I)<(r,p)(\chi).\]
\end{prop}

\begin{proof}
By Proposition \ref{prop}, there are coefficients $c\in\mathbb{R}$ and 
\[  c_{\beta}= 
     \begin{cases}
       0 \text{ if }\langle \lambda, \beta\rangle<0,\\
       -r \text{ if } \langle \lambda, \beta\rangle>0, \\
       \text{in }(-r,0] \text{ if }
\langle \lambda, \beta\rangle=0  \\ 
     \end{cases}
\] such that
\[\chi=\sum_{\beta\in\mathcal{W}} c_{\beta}\beta+ c\tau_d.\]
The representation $R(d)$ is symmetric. Consider the set \[\widetilde{I}:=\{-\beta|\,\beta\in I\}\subset \{\beta|\,\langle \lambda, \beta\rangle>0\}.\]
Then $-\sigma_I=\sum_{\widetilde{I}} \beta$. 
This means that
\[\chi-\sigma_I=\sum c'_{\beta}\beta+c\tau_d,\] 
where 
\[c'_{\beta}=\begin{cases}
  c_{\beta}+1=-r+1\text{ if }\beta\in \widetilde{I},\\
  c_{\beta}\text{ otherwise}.
\end{cases}
\]
If $-r+1>0$, replace the weight $(-r+1)\beta$ with the weight $(r-1)(-\beta)$. All the coefficients in $\chi-\sigma_I$ are thus between $-r$ and $0$, so
$\chi-\sigma_I\in r\WW.$ Because $r>\frac{1}{2}$, the set of weights $\beta$ with $c'_\beta=r$ is included in 
$\{\beta|\,\langle \lambda, \beta\rangle>0\}$; when $I$ is non-empty, this inclusion is strict. By Proposition \ref{subse}, if $\chi-\sigma_I$ lies on $F_r(\mu)^{\text{int}}$, then $\mu<\lambda$.
\end{proof}

\subsection{Admissible sets}


\subsubsection{}
\label{admissible}

In this subsection, we explain how to write a dominant weight as the sum of a       ``large part" $\chi_A$ (which appears in setting the summands of the semi-orthogonal decomposition in Theorem \ref{them1}) and a ``small part" (which is in $\frac{1}{2}\mathbb{W}$). 

We assume that $Q=Q^d$ is connected and that $Q$ is not $Q^o$. Fix $\delta\in M^{\mathfrak{S}_d}_{\mathbb{R}}$.
Let $\chi\in M^+$ and consider the standard form \eqref{godown5}:
\[\chi+\rho+\delta=-\sum_{j\in T} r_jN_j+\psi,\]
for $\psi\in \frac{1}{2}\WW$.
Let $L\cong \times_{i=1}^k G(d_i)$. Write $\chi=\sum_{i=1}^k \chi_i$ with $\chi_i\in M(d_i)^+$ for $1\leq i\leq k$ and consider the associated partition \[A=A_\chi:=(d_i,w_i)_{i=1}^k,\] where $w_i=\langle 1_{d_i}, \chi_i\rangle$. Let $S^d_w(\delta)$ be the set of partitions $A$ for which there exists $\chi\in M^+$ with $\langle 1_d, \chi\rangle=w$ and such that $A_\chi=A$. 

Let $\chi'\in M^+$ be a different weight from $\chi$ with $A_{\chi'}=A$. Consider the standard form
\[\chi'+\rho+\delta=-\sum_{j\in T'} r'_jN'_j+\psi.\]
Then $T\cong T'$, and under this identification, $\lambda_j=\lambda'_j$ and $r_j=r'_j$. Indeed, the order of the cocharacters $\lambda_j$ and the coefficients $r_j$ are computed only using the weights $w_i$ of $\chi_i$ and $\chi'_i$ for $1\leq i\leq k$. To any such partition $A$, associate the weight
\begin{equation}\label{weightA}
\chi_A:=-\sum_{j\in T} r_jN_j-\rho^{\lambda_{\dd}<0}-\delta.
\end{equation}
For $1\leq i\leq k$, consider the weights $\delta_{Ai}\in M(d_i)_{\mathbb{R}}^{\mathfrak{S}_{d_i}}$ defined by:
\begin{equation}\label{deltaai}
    -\chi_A=\sum_{i=1}^k \delta_{Ai}\text{ in }M(d)_{\mathbb{R}}\cong \bigoplus_{i=1}^k M(d_i)_{\mathbb{R}}.
    \end{equation}
To see the the weights $\delta_{Ai}$ are Weyl invariant, note that $\delta$ is $\mathfrak{S}_d$-invariant and $\rho^{\lambda_{\dd}<0}$ is a linear combination of $\mathfrak{gl}(d_i)$, and thus also $\times_{i=1}^k \mathfrak{S}_{d_i}$-invariant. The claim for $\sum_{j\in T} r_jN_j$ follows from Proposition \ref{propnj}.

The sets $S^d_w(\delta)$ and the weights $\chi_A$ are used in formulating of the semi-orthogonal decomposition from Theorem \ref{them1}. 

\begin{remark}\label{rem39}
    In general, it is not clear how to characterize the weights $\chi_A$. Given an explicit example of a quiver, one may attempt to describe the weights $\chi_A$ and to compute combinatorially the set $S^d_w(\delta)$ in order to make Theorem \ref{them1} more explicit.

    The example of $Q$ the Jordan quiver (the quiver with one vertex and one loop) is discussed in \cite{Todaquotients}, see also Subsection \ref{sub42}. Let $\chi\in M(d)^+$ with $\langle 1_d, \chi\rangle=w$ and let $\delta=-w\tau_d$. Recall the notation \eqref{notationmathfrak}. The decomposition \eqref{godown5} can be written as
    \begin{equation}\label{godown55}
    \chi+\rho-w\tau_d=-\sum_{j\in T}r_j\mathfrak{g}_j+\sum_{i=1}^k (\psi_i+\rho_i),
     \end{equation} 
    for a partition $(d_i)_{i=1}^k$ of $d$, where $\psi_i\in M(d_i)_\mathbb{R}^+$ with $
    \psi_i\in \frac{1}{2}\mathbb{W}(d_i)_0$, and $\rho_i$ is half the sum of positive roots of $\mathfrak{gl}(d_i)$. 

    Note that the only possibility for $\psi_i=0$, see \cite[Lemma 3.2]{Todaquotients}. Then $\chi_A=\chi$. So, in this case, there is a bijection between $S^d_w(\delta)$ and the integral dominant weights $\chi\in M(d)^+$ with $\langle 1_d, \chi\rangle=w$.

The example of $Q$ the quiver with one vertex and three loops has been considered in \cite{P2}, \cite{PadToda}. Let $\chi\in M(d)^+$ with $\langle 1_d, \chi\rangle=w$ and let $\delta=-w\tau_d$.
The analogue of the decomposition \eqref{godown55} is: 
\[\chi+\rho-w\tau_d=-\sum_{j\in T}3r_j\mathfrak{g}_j+\sum_{i=1}^k (\psi_i+\rho_i).\]
The weight $\chi_A$ is computed as \[\chi_A=w\tau_d-\sum_{j\in T}(3r_j-\frac{1}{2})\mathfrak{g}_j.\] It is combinatorially more convenient (as explained in loc. cit.) to consider $\widetilde{\chi}_A:=\chi_A-2\rho^{\lambda_{\dd}<0}$. Then
$\widetilde{\chi}_A=w\tau_d-\sum_{j\in T}3(r_j-\frac{1}{2})\mathfrak{g}_j$ is a dominant weight. 
Because $\widetilde{\chi}_A$ is $\times_{i=1}^k\mathfrak{S}_{d_i}$-invariant, it is a linear combination $\widetilde{\chi}_A=\sum_{i=1}^k v_i\tau_{d_i}$.
Using this, one obtains that the set $S^d_w(\delta)$ is in bijection (see the statement of \cite[Theorem 1.1]{P4}, \cite[Theorem 1.1]{PadToda}), with the set of tuplets $(d_i, v_i)_{i=1}^k\in (\mathbb{N}\times\mathbb{Z})^k$ with sum $(d,w)$ such that 
\[\frac{v_1}{d_1}<\ldots<\frac{v_k}{d_k}.\]
\end{remark}

\subsubsection{}\label{subsec332}
We assume that $Q=Q^d$ is connected and that $Q$ is not $Q^o$.
Let $\chi\in M^+$ and assume that $\chi+\rho+\delta\in \frac{1}{2}\partial\WW$. Using Corollary \ref{godown}, write
\[\chi+\rho+\delta=-\frac{1}{2}N^{\lambda_{\dd}>0}+\psi,\]
for $\psi\in \frac{1}{2}\WW^o$ and a partition $\dd=(d_i)_{i=1}^k$.
Write $\chi=\sum_{i=1}^k \chi_i$ with $\chi_i\in M(d_i)^+$ for $1\leq i\leq k$ and consider the associated partition \[A=A_\chi:=(d_i,w_i)_{i=1}^k,\] where $w_i=\langle 1_{d_i}, \chi_i\rangle$. Let $T^d_w(\delta)$ be the set of partitions $A$ for which there exists $\chi\in M^+$ with $\langle 1_d, \chi\rangle=w$ such that $A_\chi=A$. We define $\chi_A$ as in the previous Subsection. 

Let $U^d_w(\delta)$ be the set of partitions $A=(e_i, v_i)_{i=1}^l$ for which there exists a partition $B=(d_i, w_i)_{i=1}^k$ in $S^d_w(\delta)$,  integers \[a_0=0< a_1<\cdots<a_{k-1}\leq a_k=l\] such that for any $1\leq j\leq k$, $(e_i, v_i)_{i=a_{j-1}+1}^{a_{j}}$ is a partition in $T^{d_{j}}_{w_{j}}(\delta_{Bj})$.

\begin{remark}

We continue the discussion from Remark \ref{rem39}.

If $Q$ is the Jordan quiver, then $T^d_w(\delta)$ is empty.

If $Q$ is the quiver with one vertex and three loops, then $T^d_w(\delta)$ is in bijection with the set of tuplets $(d_i, v_i)_{i=1}^k\in (\mathbb{N}\times\mathbb{Z})^k$ with sum $(d,w)$ such that 
\[\frac{v_1}{d_1}=\ldots=\frac{v_k}{d_k}.\]
\end{remark}

\subsubsection{}\label{swe}

We assume that $Q=Q^d$.

Assume first that $Q=Q^o$. Then $\delta$ is a multiple of $\tau_d$. Let $U^d_w(\delta)=S^d_w(\delta)$ be the set of partitions $(1,w_i)_{i=1}^d$ of $(d,w)\in\mathbb{N}\times\mathbb{Z}$ with $w_1\geq\cdots\geq w_d$. Let $T^d_w(\delta)$ be empty for $d>1$ and $T^1_w(\delta)=\{(1,w)\}.$

Assume that $Q$ is a disconnected quiver and let $\delta\in M(d)_{\mathbb{R}}^{\mathfrak{S}_d}$.
If $Q$ is a disjoint union of connected quivers $Q_j$ for $j\in J$, write $d_j$ and $\delta_j$ for the corresponding dimension vector and Weyl invariant weight of $Q_j$ for $j\in J$. Let $C$ be the set of partitions $(w_j)_{j\in J}$ of $w$.
Let 
\begin{align*}
    S^d_w(\delta)&:=\bigsqcup_C\left(\times_{j\in J} S^{d_j}_{w_j}(\delta_j)\right),\\
    T^d_w(\delta)&:=\bigsqcup_C\left(\times_{j\in J} T^{d_j}_{w_j}(\delta_j)\right).
\end{align*}
Then $U^d_w(\delta)\cong\bigsqcup_C
\left(\times_{j\in J} U^{d_j}_w(\delta_j)\right).$

\subsubsection{}\label{compadm}
Assume first that $Q=Q^d$ is connected and that $Q$ is not $Q^o$. Consider two partitions $A=(d_i, w_i)_{i=1}^k$ and $B=(e_i, v_i)_{i=1}^l$ in $S^d_w(\delta)$. Let $\chi_A$ and $\chi_B$ be two weights with associated sets $A$ and $B$:
\begin{align*}
    \chi_A+\rho+\delta&:=-\sum_{j\in T_A}r_{A,k}N_{A,k}+\psi_A,\\
    \chi_B+\rho+\delta&:=-\sum_{j\in T_B}r_{B,k}N_{B,k}+\psi_B.
\end{align*}
The set $R\subset S^d_w(\delta)\times S^d_w(\delta)$ contains pairs $(A,B)$ for which there exists $c\geq 1$ such that $r_{A,c}>r_{B,c}$ and $r_{A,i}=r_{B,i}$ for $i<c$, or for which there exists $c\geq 1$ such that $r_{A,i}=r_{B,i}$ for $i\leq c$, $\lambda_{A,i}=\lambda_{B,i}$ for $i<c$, and $\lambda_{A,c}> \lambda_{B,c}$, or with $A=B$.

Let $O:=S^d_w(\delta)\times S^d_w(\delta)\setminus R$.

For $Q^o$, let $R=\{(A,A)|A\in S^d_w(\delta)\}$, and let $O:=S^d_w(\delta)\times S^d_w(\delta)\setminus R$.

Assume $Q$ is a disconnected quiver. We continue with the notation from Subsection \ref{swe}.
Consider the set $R_{j}\subset S^{d_j}_{w_j}(\delta_j)\times S^{d_j}_{w_j}(\delta_j)$ for the quiver $Q_j$, let \[R:=\bigsqcup_{C} \left(\times_{j\in J} R_j\right)\subset S^d_w(\delta)\times S^d_w(\delta),\] and let $O:=S^d_w(\delta)\times S^d_w(\delta)\setminus R$.

\section{Semi-orthogonal decompositions and relations in the categorical Hall algebra}\label{s3}

\subsection{Semi-orthogonal decompositions}\label{subsec41}


In this Section, we prove Theorem \ref{them1}. 
Let 
\begin{align*}
   \mathbb{M}_A(\delta)&:=\otimes_{i=1}^k \mathbb{M}\left(d_i;\delta_{Ai}\right)_{w_i}, \\
   \mathbb{S}_A(\delta)&:=\text{MF}\left(\mathbb{M}_A(\delta), \bigoplus_{i=1}^k W_{d_i}\right).
\end{align*}
Given a grading, recall that $\mathbb{S}^{\text{gr}}_A(\delta):=\text{MF}^{\text{gr}}\left(\mathbb{M}_A(\delta), \bigoplus_{i=1}^k W_{d_i}\right)$. We first discuss the generation statement.

\begin{prop}\label{gen0}
The categories $p_{\dd*}q_{\dd}^*\left(\mathbb{M}_A(\delta)\right)$ for partitions $A\in S^d_w(\delta)$ generate $D^b(\X(d))_w$. 
\end{prop}

\begin{proof}
We may assume that $Q=Q^d$ is connected.

For $Q^o$, let $A=(1,w_i)_{i=1}^d$ with $w_1\geq\cdots\geq w_d$. Let $\chi_A:=\sum_{i=1}^d w_i\beta_i$.
We have that \[\mathbb{M}(1)_w\cong \mathcal{O}_{\X(1)}(w)\]in $D^b\text{Coh}(\X(1))\cong D^b\text{Rep}(\mathbb{C}^*).$ Then $p_{\dd*}q_{\dd}^*\left(\mathbb{M}_A(\delta)\right)$ is the category generated by $\Gamma_{GL(d)}(\chi)$. These vector bundles generate $\text{Coh}(\X(d))\cong \text{Rep}(GL(d))$ and thus $D^b(\X(d))$. 

We next assume that $Q$ is different from $Q^o$.
Let $\chi$ be a dominant weight for $G(d)=:G$ and let $w=\langle 1_d, \chi\rangle$.
It is enough to show that the sheaf $\OO_{\X(d)}\otimes \Gamma_G(\chi)$ is generated by the categories $p_{\dd*}q_{\dd}^*\left(\mathbb{M}_A(\delta)\right)$ for partitions $A\in S^d_w(\delta)$. 
If $r(\chi+\rho+\delta)\leq\frac{1}{2}$, then $\OO_{\X(d)}\otimes \Gamma_G(\chi)$ is in $\mathbb{M}(d)_w$. Assume that $r(\chi+\rho+\delta)>\frac{1}{2}$. The set of $(r,p)$-invariants less than a fixed pair is finite.
We use induction on $(r,p)(\chi+\rho+\delta)$.
By Proposition \ref{prop:decompchi}, we have \[\chi+\rho+\delta=-\sum_{j\in T} r_jN_j+\psi,\]
where the sum is taken over a path of partitions $T$ of antidominant cocharacters, $r_j>\frac{1}{2}$, 
$\psi \in \frac{1}{2}\mathbb{W}(L)$, and $L\cong\times_{i=1}^k G(d_i)$. Let $\lambda=\lambda_{\dd}$.
The weight $\psi$ is $L$-dominant. Let $\langle 1_d, \psi\rangle=:v$, then $v=w+\langle 1_d, \delta_d\rangle$.

Write $\chi=\sum_{i=1}^k\chi_i$ in $M(d)_{\mathbb{R}}\cong \bigoplus_{i=1}^k M(d_i)_{\mathbb{R}}$. The weight $\chi$ is in $M(d)^+$, so $\chi_i\in M(d_i)^+$ for every $1\leq i\leq k$. Let $w_i=\langle 1_{d_i}, \chi_i\rangle$. 
Write $\tau_d=\sum_{i=1}^k \widetilde{\tau_{i}}$, then $\widetilde{\tau_{i}}$ is a multiple of $\tau_{d_i}$. 
Then $\psi=\sum_{i=1}^k \left(\psi_i+v\widetilde{\tau_{i}}\right)$, where $\psi_i\in\frac{1}{2}\WW(d_i)_0$. Consider weights $\delta_{i}\in M(d_i)_{\mathbb{R}}$ by
\begin{equation*}
    \rho^{\lambda<0}+\delta+\sum_{j\in T}r_jN_j=\sum_{i=1}^k \delta_{i}\text{ in }M(d)_{\mathbb{R}}\cong \bigoplus_{i=1}^k M(d_i)_{\mathbb{R}}.
    \end{equation*}
We claim that actually $\delta_i\in M(d_i)^{\mathfrak{S}_{d_i}}_\mathbb{R}$. It suffices to check that the weight $\rho^{\lambda<0}+\delta+\sum_{j\in T}r_jN_j$ is $\times_{i=1}^k \mathfrak{S}_{d_i}$-invariant. 
Note that $\delta$ is $\mathfrak{S}_d$-invariant and $\rho^{\lambda<0}$ is a linear combination of $\mathfrak{gl}(d_i)$, and thus also $\times_{i=1}^k \mathfrak{S}_{d_i}$-invariant. The claim for $\sum_{j\in T} r_jN_j$ follows from Proposition \ref{propnj}.
    
For $1\leq i\leq k$, we have that
\[\chi_i+\rho_i+\delta_i=\psi_i+v\widetilde{\tau_{i}}\in \frac{1}{2}\WW(d_i).\] 
Then $A=A_\chi:=(d_i,w_i)_{i=1}^k$ is in $S^d_w(\delta)$, the weights $\delta_{Ai}$ are $\delta_i$ for $1\leq i\leq k$, and
so $\mathcal{O}_{\X(d_i)}\otimes\Gamma_{G(d_i)}(\chi_i)\in \mathbb{M}(d_i; \delta_i)_{w_i}$.  

Consider the complex $\mathcal{O}_{\X(d)^\lambda}\otimes \Gamma_L(\chi)$ in $\mathbb{M}_A$. By Proposition \ref{bbw}, there is a spectral sequence with terms $\OO_{\X(d)}\otimes \Gamma_G\left(\left(\chi-\sigma_I\right)^+\right)$ converging to \[p_{\dd*}q_{\dd}^*\left(\mathcal{O}_{\X(d)^\lambda}\otimes \Gamma_L(\chi)\right),\] where $\sigma_I$ is a sum of weights in $I\subset \{\beta|\,\langle \lambda,\beta\rangle<0\}$.
By Propositions \ref{rem1} and \ref{rgoesdown}, we have that \[(r,p)\left(\left(\chi+\rho-\sigma_I+\delta\right)^+\right)=(r,p)(\chi+\rho-\sigma_I+\delta)<(r,p)(\chi+\rho+\delta)\] for $I$ a non-empty set. The conclusion thus follows.
\end{proof}

Before we discuss the orthogonality statement, we note a preliminary computation.

\begin{prop}\label{zerodc}
Let $G$ be a reductive group, $V$ a $G$-representation with origin $0$, $\lambda$ a cocharacter, and $w\in \mathbb{Z}$. Let $\X=V/G$ be the quotient stack. Assume that $\mathcal{F}$ and $\mathcal{E}$ are in $D^b\left(\X^{\lambda\geq 0}\right)$ such that $\mathcal{F}|_0$ has $\lambda$-weights $> w$ and $\mathcal{E}|_0$ has $\lambda$-weights $\leq w$. Then 
\begin{equation}\label{extzero}
    \text{Ext}^a_{\X^{\lambda\geq 0}}\left(\mathcal{F}, \mathcal{E}\right)=0\end{equation} for any $a\in\mathbb{Z}$. 
\end{prop}

\begin{proof}
This statement is contained in \cite[Amplification 3.18]{hl}.
We briefly explain a proof. We can assume that $\mathcal{F}$ and $\mathcal{E}$ are locally free sheaves.
The vanishing \eqref{extzero} is then clear for $a\neq 0$. 

Let $\mathcal{Z}=V^{\lambda\geq 0}/\mathbb{C}^*$. 
For $a=0$, it suffices to check the statement for locally free sheaves on $\mathcal{Z}$.
We can assume that $\mathcal{F}=\OO_{\mathcal{Z}}(u)$ and $\mathcal{E}=\OO_{\mathcal{Z}}(v)$ with $v\leq w<u$. 
We have that:
\[\text{Hom}_{\mathcal{Z}}
\left(\OO_{\mathcal{Z}}(u), \OO_{\mathcal{Z}}(v)\right)\cong
\left(\mathbb{C}\left[\left(V^{\lambda\geq 0}\right)^{\vee}\right](v-u)\right)^{\mathbb{C}^*}=0.\] 
\end{proof}

We next discuss the orthogonality statement.

\begin{prop}\label{orth} 
Recall the set $O\subset U^d_w(\delta)\times U^d_w(\delta)$ from Subsection \ref{compadm}. 
Consider partitions $A=(d_i, w_i)_{i=1}^k$, $B=(e_i, v_i)_{i=1}^l$ in $S^d_w(\delta)$ with $(A,B)\in O$ and with associated antidominant cocharacters $\lambda_A$ and $\lambda_B$. Denote by $\X^A=\X(d)^{\lambda_A}$ and $\X^B=\X(d)^{\lambda_B}$.
Let $\mathcal{F}_A, \mathcal{F}'_A\in 
\mathbb{M}_A$, $\mathcal{F}_B \in
\mathbb{M}_B$, and $a\in\mathbb{Z}$. 
Then
\begin{align*}
    \text{Ext}^a_{\X(d)}\left(p_{\dd*}q_{\dd}^*\mathcal{F}_A, p_{\ee*}q_{\ee}^*\mathcal{F}_B\right)&=0,\\
    \text{Ext}^a_{\X(d)}\left(p_{\dd*}q_{\dd}^*\mathcal{F}_A, p_{\dd*}q_{\dd}^*\mathcal{F}'_A\right)&\cong\text{Ext}^a_{\X^A}\left(\mathcal{F}_A, \mathcal{F}'_A\right).
\end{align*}
\end{prop}

\begin{proof}
We may assume that $Q=Q^d$ is connected.

If $Q=Q^o$, then $d_i=e_i=1$. Let $\chi_A:=\sum_{i=1}^k w_i\beta_i$ and $\chi_B:=\sum_{i=1}^k v_i\beta_i$. The categories $\text{Coh}(\X(d))\cong \text{Rep}(G(d))$ and $\text{Coh}(\X(1)^{\times d})\cong \text{Rep}(T(d))$ are semisimple, so we can assume that $a=0$. Let $G:=G(d)$ and $T:=T(d)$. We can also assume that $\mathcal{F}_A\cong\mathcal{F}'_A\cong \Gamma(\chi_A)$ and $\mathcal{F}_B\cong \Gamma(\chi_B)$. Then
\begin{align*}
    \text{Hom}_{\text{Rep(G)}}(\Gamma_G(\chi_A), \Gamma_G(\chi_B))&\cong 0,\\
    \text{Hom}_{\text{Rep(G)}}(\Gamma_G(\chi_A), \Gamma_G(\chi_A))&\cong \text{Hom}_{\text{Rep(T)}}(\Gamma_T(\chi_A), \Gamma_T(\chi_A))\cong \mathbb{C}.
\end{align*}

We assume next that $Q$ is not $Q^o$.
Denote the Levi group of $\lambda_A$ by $L$, the Levi group of $\lambda_B$ by $H$, and let $G:=G(d)$.
It is enough to prove the result for generators of the above categories:
\begin{align*}
    \mathcal{F}_A&=\Gamma_L(\chi_A)\otimes\OO_{\X^A},\\
    \mathcal{F}'_A&=\Gamma_L(\chi'_A)\otimes\OO_{\X^A},\\
    \mathcal{F}_B&=\Gamma_H(\chi_B)\otimes\OO_{\X^B}.
\end{align*}
Let $\mu_A$ and $\mu_B$ be antidominant cocharacters such that
\begin{align*}
    \chi_A+\rho+\delta, \chi'_A+\rho+\delta&\in F_r(\mu_A)^{\text{int}}\subset F_r(\lambda_A),\\
    \chi_B+\rho+\delta&\in F_s(\mu_B)^{\text{int}}\subset F_r(\lambda_A).
\end{align*}
Note that $\mu_A\geq \lambda_A$ and $\mu_B\geq \lambda_B$ in the notation of Subsection \ref{subseccompa}, see Proposition \ref{prop}. 
In particular, if $\beta$ is a weight such that $\langle \lambda_A, \beta\rangle>0$, then $\langle \mu_A, \beta\rangle>0$ and similarly for $B$.

Assume for simplicity that $s>r$ or $s=r$ and $\mu_A$ is not $\geq\mu_B$, which corresponds to the case $c=1$ in the definition of the comparison set $O$ from Subsection \ref{compadm}. The general case follows similarly.

For the first statement, by adjunction one needs to prove
\begin{equation}\label{parta}
    \text{Ext}^a_{\X(d)^{\lambda_B\geq 0}}\left(p_{\ee}^*p_{\dd*}q_{\dd}^*\mathcal{F}_A, q_{\ee}^*\mathcal{F}_B\right)=0.
    \end{equation}
We have that $\chi_B+\rho+\delta\in F_s(\mu)^{\text{int}}$. By Propositions \ref{bbw} and \ref{rgoesdown}, there is a spectral sequence with vector bundles $\OO_{\X(d)}\otimes \Gamma_G\left(\left(\chi_A-\sigma_I\right)^+\right)$ converging to $p_{\dd*}q_{\dd}^*\mathcal{F}_A$ with $\chi_A+\rho-\sigma_I+\delta\in r\mathbb{W}$, where $\sigma_I$ is the sum of weights in $I\subset \{\beta|\,\langle \lambda_A, \beta\rangle<0\}$.
Further, if $\left(\chi_A+\rho-\sigma_I+\delta\right)^+ \in F_r(\omega)^{\text{int}}$, then $\omega\leq \lambda_A$. 
For any such set $I$, we have that either $s>r$ or $s=r$ and $\omega$ is not $\geq \mu_B$, and so
\[\big\langle \mu_B,  \left(\chi_A-\sigma_I\right)^+\big\rangle >\langle \mu_B, \chi_B\rangle.\] The statement in \eqref{parta} follows from Proposition \ref{zerodc}.

For the second statement, we need to show that
\begin{equation}\label{partb}
    \text{Ext}^a_{\X(d)^{\lambda_A\geq 0}}\left(p_{\dd}^*p_{\dd*}q_{\dd}^*\mathcal{F}_A, q_{\dd}^*\mathcal{F}'_A\right)\cong \text{Ext}^a_{\X^A}\left(\mathcal{F}_A, \mathcal{F}'_A\right).
    \end{equation}
  We have that $\chi'_A+\rho+\delta\in F_r(\mu_A)^{\text{int}}$. We use the notations from the previous paragraph. By Proposition \ref{rgoesdown}, 
  if $\left(\chi+\rho-\sigma_I+\delta\right)^+ \in F_r(\omega)^{\text{int}}$, then $\omega\leq \lambda_A$ and equality holds only for $I$ empty. 
  For any non-empty $I$, we have that
\[\big\langle \mu_A,  \left(\chi_A-\sigma_I\right)^+\big\rangle >\langle \mu_A, \chi'_A\rangle,\] and the statement in \eqref{partb} follows from Proposition \ref{zerodc}.
\end{proof}


\begin{proof}[Proof of Theorem \ref{them1}]
Assume $W=0$. Then the semi-orthogonal decomposition follows from Propositions \ref{gen0} and \ref{orth}. 

Assume $W$ is arbitrary.
Let $(d_i,w_i)_{i=1}^k$ be a partition of $(d,w)$.
Let $\mathbb{D}$ be a subcategory of $\otimes_{i=1}^k D^b\left(\X(d_i)\right)_{w_i}$ on which $p_{\dd*}q_{\dd}^*$ is fully faithful. Then 
\begin{multline*}
    \text{MF}\left(p_{\dd*}q_{\dd}^*\,\mathbb{D}, W_d\right)\cong p_{\dd*}\text{MF}\left(q^*_{\dd}\,\mathbb{D}, p_{\dd}^*W_d\right)\cong p_{\dd*}q_{\dd}^*\,\text{MF}\left(\mathbb{D}, \oplus_{i=1}^kW_{d_i}\right)\\
    \cong \text{MF}\left(\mathbb{D}, \oplus_{i=1}^kW_{d_i}\right)
    .\end{multline*} 
The semi-orthogonal decomposition for $\text{MF}(\X(d), W)$ follows from the semi-orthogonal decomposition for $W=0$ and \cite[Proposition 2.1]{P1}. The argument for $\text{MF}^{\text{gr}}(\X(d), W)$ is similar and follows from \cite[Proposition 2.2]{P1}.
\end{proof}

\subsection{The Jordan quiver}\label{sub42}
The Jordan quiver is the quiver with one vertex and one loop.
 Consider the Hall algebra for the Jordan quiver $J$ and zero potential:
\[\text{HA}(J,0):=\bigoplus_{d\in\mathbb{N}}D^b\left(\X(d)\right).\] 
The categories $\mathbb{M}(d;\delta_d)_w$ are independent of $\delta_d$, so we drop it from the notation.
The category $\mathbb{M}(d)_{dv}\subset D^b\left(\X(d)\right)_{dv}$ is generated by the line bundles $\mathcal{O}_{\X(d)}(v)$ for $v\in\mathbb{Z}$. In particular, $\mathbb{M}(d)_{dv}\cong D^b\left(\text{pt}\right)$. The semi-orthogonal decomposition in Theorem \ref{them1} is
\[D^b\left(\mathcal{X}(d)\right)_w=\big\langle \otimes_{i=1}^k\mathbb{M}(d_i)_{d_iv_i}\big\rangle,\]
where the right hand side is after all partitions $A=(d_i, d_iv_i)_{i=1}^k$ of $(d,w)$ with $v_i>v_j$ for every $1\leq i<j\leq k$. The corresponding weight of $A$ is 
\[\chi_A=\sum_{i=1}^k v_i\nu_{d_i}\in M(d)^+.\]
In $K_0(\X(d))$, we have that:
\[ \left[\mathcal{O}_{\X(1)}(v)\right]^d=d!\left[\mathcal{O}_{\X(d)}(v)\right].\]
In $\text{KHA}(J,0)_{\mathbb{Q}}$, the generators are $ K_0(\X(1))_{\mathbb{Q}}$. 

For a discussion of Theorem \ref{them1} for preprojective Hall algebras and for examples in those cases, see \cite{P4}.

\subsection{Relations in the Hall algebra}

\subsubsection{}\label{subsec431}
Let $(d, w)\in\mathbb{N}^I\times\mathbb{Z}$ and $\delta\in M^{\mathfrak{S}_d}_{\mathbb{R}}$.
Consider the set \[B_{d,w}=\left\{\frac{1}{2}, (r,p)(A)\Big\rvert \,A\in S^d_w(\delta)\right\}.\] Let $\beta: B_{d,w}\to \mathbb{N}\setminus \{0\}$ be the order preserving bijection. 
Define a filtration $F^{\leq i}_{d,w}$ on $D^b(\X(d))_w$ containing $\mathbb{M}_A$ for $\beta\left((r,p)(A)\right)\leq i$.
Then $F^{\leq 1}_{d,w}:=\mathbb{M}(d; \delta)_w$. The filtration depends on $\delta$.
It is not clear whether we can choose $\delta\in M^{\mathfrak{S}_d}_{\mathbb{R}}$ for all $d\in\mathbb{N}^I$ such that these filtrations are compatible with multiplication. However, once we fix a partition $A=(d_i,w_i)_{i=1}^k$ in $S^d_w(\delta)$, the filtrations on $D^b(\X(d_i))_{w_i}$ for the weights $\delta_{Ai}$ are compatible in the following sense.
Consider integers $\ell_j\geq 1$ for $1\leq j\leq k$ with sum $\ell$ such that $\ell_a\geq 2$ for some $1\leq a\leq k$. Let $\alpha$ be an integer such that
\[F^{\leq \ell_1}_{d_1, w_1}\cdots  F^{\leq \ell_a}_{d_a,w_a}\cdots F^{\leq \ell_k}_{d_k, w_k} \subset F^{\leq \ell+\alpha}_{d,w}.\]
Then 
\[F^{\leq \ell_1}_{d_1, w_1}\cdots  F^{\leq \ell_a-1}_{d_a,w_a}\cdots F^{\leq \ell_k}_{d_k, w_k} \subset F^{\leq \ell-1+\alpha}_{d,w}.\]
We define the analogous filtrations for $\text{MF}$ and $\text{MF}^{\text{gr}}$ with the same property as above.

\subsubsection{}\label{subsec432}
Let $A\in S^d_w(\delta)$. Consider the subcategory $D^b(\X(d))_{<A}$ of $D^b(\X(d))$ generated by the categories $\mathbb{M}_B(\delta)$ for $(r,p)(B)<(r,p)(A)$. Let $D^b(\X(d))_{\leq A}$ be the subcategory of $D^b(\X(d))$ generated by $\mathbb{M}_A(\delta)$ and $D^b(\X(d))_{<A}$. 
By Theorem \ref{them1} for $W=0$, there is a semi-orthogonal decomposition
\[D^b(\X(d))_{\leq A}=\big\langle \mathbb{M}_A(\delta), D^b(\X(d))_{< A}\big\rangle.\] 
We denote by $\Phi_A: D^b(\X(d))_{\leq A}\to \mathbb{M}_A(\delta)$ the left adjoint to the natural inclusion. Denote by $\text{MF}(\X(d),W)_{\leq A}$ and $\text{MF}^{\text{gr}}(\X(d),W)_{\leq A}$
the analogous categories for a general potential.

\subsubsection{}\label{gammadelta}

For $e, f\in\mathbb{N}^I$, let $\lambda_{e,f}$ be an antidominant cocharacter corresponding to the partition $(e,f)$ of $e+f$, see Subsection \ref{paco}. Consider the set $\mathcal{W}_{e, f}=\{\beta|\,\langle \lambda_{e, f}, \beta\rangle<0\}$. We use the notations $L_{e, f}:=L^{\lambda_{e, f}<0}$, $N_{e,f}:=N^{\lambda_{e,f}<0}$. Let $\rho_{e,f}:=\rho^{\lambda_{e,f}<0}$ be half the sum of positive roots pairing negatively with $\lambda_{e,f}$. Then 
\[L_{e,f}=N_{e,f}-2\rho_{e,f}.\]
Recall the notations from \eqref{genK}.
Define the equivariant monomials
\[q^{\gamma(e,f)}\in K_0(BG(e))\text{ and }
q^{-\delta(e,f)}\in K_0(BG(f))\] by the equality:
\[L_{e,f}=q^{\gamma(e,f)}q^{-\delta(e,f)}\in K_0\left(BG(e)\times BG(f)\right).\]
Denote also by $N_{e,f}$ and $L_{e,f}$ the corresponding $G(e)\times G(f)$-equivariant line bundles.
Let $w_{e,f}$ be the Weyl group element \[w_{e,f}=(w^a)_{a\in I}\in\mathfrak{S}\cong \prod_{a\in I}\mathfrak{S}_{e^a+f^a},\] such that for a vertex $a\in I$: 
\[w^a(i)=\begin{cases}
  i+f, \text{ if }1\leq i\leq e,\\
  i-e, \text{ if }e<i\leq e+f.
\end{cases}
\]
Then $w_{f,e}=w_{e,f}^{-1}$.

\begin{prop}\label{Wcomputation}
(a) We have that $w_{e,f}\cdot \rho-\rho=-2\rho_{f,e}$.

(b) Multiplication by $-w_{f,e}$ gives a bijection of sets
$(-w_{f,e})\cdot \mathcal{W}_{f,e}\cong
\mathcal{W}_{e,f}.$
\end{prop}

\begin{proof}
To simplify notation for $(a)$, assume that there is only one vertex. The weights contributing to $w_{e,f}\cdot\rho$ which are not positive are of the form 
$\beta_{i+f}-\beta_{j-e}$ with  $1\leq i\leq e<j\leq e+f$.
Then 
\[w_{e,f}\cdot \rho-\rho=\sum_{i\leq e<j}\left(\beta_{i+f}-\beta_{j-e}\right)=-2\rho_{f,e}.\]
For $(b)$, the weights in $\mathcal{W}_{f,e}$ are of the form $\beta^a_{i}-\beta^b_{j}$ where $a$ and $b$ be vertices with an edge between them, $1\leq i\leq f^a$, and $f^b+1\leq j\leq e^b+f^b$. 
Then
\[\left(-w_{f,e}\right)\cdot\left(\beta^a_{i}-\beta^b_{j}\right)=\beta^b_{j-f^b}-\beta^a_{i+e^a}.\]
The weight $\beta^b_{j-f^b}-\beta^a_{i+e^a}$ is in $\mathcal{W}_{e,f}$ because there is an edge between $b$ and $a$ as $Q$ is symmetric, $1\leq j-f^b\leq e^b$, and $e^a+1\leq i+e^a\leq e^a+f^a$.
\end{proof}

\subsubsection{}
 
In this subsection, we denote by $A$ a partition in $S^d_w(\delta_d)$ with parts $(e,v), (f,u)$ and associated weights $\delta_e$ and $\delta_f$, see \eqref{deltaai}.

\begin{prop}\label{commu2}
 Consider vector bundles  $\mathcal{E}=\OO_{\X(e)}\otimes \Gamma_{G(e)}(\chi_e)\in \mathbb{M}(e; \delta_e)_v$ and $\mathcal{F}=\OO_{\X(f)}\otimes \Gamma_{G(f)}(\chi_f)\in \mathbb{M}(f; \delta_f)_u$. 
 
 (a) We have that \[p_{f,e*}q_{f,e}^*\big(\mathcal{F}\boxtimes \mathcal{E}\otimes L_{f,e}[\chi(e,f)]\big)\in D^b(\X)_{\leq A}.\]
 
(b) There is a natural isomorphism
 \begin{equation*}\label{commu}
    \Phi_A\left(p_{f,e*}q_{f,e}^*\left(\mathcal{F}\boxtimes \mathcal{E}\otimes L_{f,e}[\chi(e,f)]\right)\right)\xrightarrow{\sim} \mathcal{E}\boxtimes\mathcal{F}.
\end{equation*}

(c) There is a natural map
\[p_{f,e*}q_{f,e}^*\left(\mathcal{F}\boxtimes \mathcal{E}\otimes L_{f,e}[\chi(e,f)]\right)\rightarrow p_{e,f*}q_{e,f}^*\left(\mathcal{E}\boxtimes\mathcal{F}\right)\] and its cone is in $D^b(\X)_{<A}$. 

\end{prop}

\begin{proof}
We may assume that $Q=Q^d$ is connected.

If $Q=Q^o$, then $d=e=1$ and the statements follow from the Borel-Weyl-Bott theorem. 

We assume next that $Q$ is not $Q^o$.

(a)
We may assume that \begin{align*}
    \mathcal{E}&=\OO_{\X(e)}\otimes \Gamma_{G(e)}(\chi_e),\\
    \mathcal{F}&=\OO_{\X(f)}\otimes\Gamma_{G(f)}(\chi_f).
\end{align*} Let $H:=G(f)\times G(e)$.
By Proposition \ref{bbw}, there exists a spectral sequence with terms 
\[\OO_{\X(d)}\otimes \Gamma_{G(d)} \left(\left(\chi_f+\chi_e+L_{f,e}-\sigma_I\right)^+\right)\] in degree $|I|-\ell(I)$ converging to $p_{f,e*}q_{f,e}^*\left(\Gamma_H(\chi_f+\chi_e)\otimes L_{f,e}\otimes\OO_{\X(e)\times\X(f)}\right),$ where $I\subset \mathcal{W}_{f,e}$. Recall that $L_{f,e}=N_{f,e}-2\rho_{f,e}$. 
Using Corollary \ref{godown}, write
\[\chi_e+\chi_f+\rho+\delta_d=rN_{e,f}+\psi\] for $r>\frac{1}{2}$ and $r(\psi)<r$. 
Using Proposition \ref{Wcomputation}, we have that
\[\chi_f+\chi_e+N_{f,e}-2\rho_{f,e}+\rho+\delta_d=(1-r)N_{f,e}+w_{e,f}\psi,\]
and thus for $I\subset \mathcal{W}_{f,e}$ we have that
\[\chi_f+\chi_e+N_{f,e}-\sigma_I-2\rho_{f,e}+\rho+\delta_d=(1-r)N_{f,e}-\sigma_I+w_{e,f}\psi.\]
Denote the left hand side by $\theta$. 
Using Proposition \ref{subse}, we have $r(\theta)\leq r$ and if $\theta\in r\partial\WW$, then $I=\mathcal{W}_{f,e}$.

To show (b), observe first that the corresponding highest weight for the partial sum $I=\mathcal{W}_{f,e}$ is
\[w_{f,e}*\left(\chi_f+\chi_e+N_{f,e}-2\rho_{f,e}-N_{f,e}\right)=\chi_e+\chi_f.\]
The corresponding shift is $|\mathcal{W}_{f,e}|-\ell\left(\mathcal{W}_{f,e}\right)=-\chi(e,f)$.
There are thus natural isomorphisms
\[\Phi_A\left(p_{f,e*}q_{f,e}^*\left(\mathcal{F}\boxtimes \mathcal{E}\otimes L_{f,e}[\chi(e,f)] \right)\right)\xrightarrow{\sim} \mathcal{E}\boxtimes\mathcal{F}\]
for $\mathcal{E}=\OO_{\X(e)}\otimes \Gamma_{G(e)}(\chi_e)$ and $\mathcal{F}=\OO_{\X(f)}\otimes\Gamma_{G(f)}(\chi_f)$. 
Part (c) follows from (b) using the adjoint pair of functors
$\Phi_A: D^b(\X(d))_{\leq A}
\rightleftarrows \mathbb{M}_A(\delta): p_{*}q^*.$
\end{proof}


\section{The PBW theorem for KHAs}\label{s4}

\subsection{Preliminaries} In this Subsection, we assume that $Q=Q^d$. 
\subsubsection{}\label{subsec511}
Assume first that $Q=Q^d$ is connected.
Consider a partition $A\in T^d_w(\delta)$ with terms $(e,v)$ and $(f,u)$ and associated weights $\delta_e$ and $\delta_f$ such that
\[-\chi_A=\delta_e+\delta_f,\] see \eqref{deltaai}. 
Let $\lambda=\lambda_{e,f}$ and $m:=-n_\lambda/2-\langle \lambda, \delta\rangle$.
By the definition of the category $\mathbb{M}(d;\delta)$ and by \cite[Corollary 3.28]{hl}, the functor $p_\lambda^*$ has image in
\[p_\lambda^*: \mathbb{M}(d; \delta)\to \Big\langle q_\lambda^*\left(D^b\left(\X(e)\times\X(f)\right)_w\right)\Big\rangle,\]
where the right hand side contains categories for all weights $w\geq m$. Denote the right hand side by $\mathbb{E}$.
For all summands of $\mathbb{E}$, 
the functor $q_\lambda^*$ induces an equivalence $D^b\left(\X(e)\times\X(f)\right)_w\cong q_\lambda^*\left(D^b\left(\X(e)\times\X(f)\right)_w\right)$. 
Consider the adjoint to the natural inclusion $q_\lambda^*: D^b\left(\X(e)\times\X(f)\right)_w\hookrightarrow \mathbb{E}$:
\[\beta_m: \mathbb{E}\to D^b\left(\X(e)\times\X(f)\right)_m.\]
The composition $\beta_m p_\lambda^*$ induces a functor
\[\Delta_A:=\beta_m p_\lambda^*: \mathbb{M}(d; \delta)_w\to \mathbb{M}(e; \delta_e)_v\otimes \mathbb{M}(f; \delta_f)_u=:\mathbb{M}_A.\]
It induces a functor 
\[\Delta_A: \mathbb{S}(d;\delta)_w\to \mathbb{S}_A(\delta).\]
Recall the order on the set $T^d_w(\delta)$ defined as in Subsection \ref{compa}. Let $B>A$ be in $T^d_w(\delta)$. The above construction also provides a functor
\[\Delta_{AB}: \mathbb{S}_A(\delta)\to \mathbb{S}_B(\delta).\]
If $Q$ is not connected, define the functors $\Delta$ in the natural way using the decompositions in Subsection \ref{swe}. 

When $Q=Q^o$, the constructions are non-zero only for $A$ and $B$ the length $1$ partitions $(1,w)$ and $\Delta=\text{id}$.

\subsubsection{} 
Recall the equivalence $D_{\text{sg}}(\X(d)_0)\cong \text{MF}(\X(d), W)$ and the exact sequence \eqref{es}. 
By \cite[Proposition 3.6]{P1}, the pushforward $i_*:\X(d)_0\to \X(d)$ induces an algebra morphism
\[i_*: \text{KHA}(Q,W)\to \text{KHA}(Q,0).\] It can be also described using matrix factorizations \begin{align*}
    i_*: K_0(\text{MF}(\X(d),W))&\to K_0(\X(d))\\
    \left[\left(\alpha: F\rightleftarrows G: \beta\right)\right]&\mapsto [F]-[G].
\end{align*}
It also induces a map
\begin{equation}\label{imim}
    i_*: K_0\left(\mathbb{S}_A(\delta)\right)\to K_0\left(\mathbb{M}_A(\delta)\right).\end{equation}
    Denote the image of \eqref{imim} by $K_0\left(\mathbb{S}_A(\delta)\right)'$. 
\begin{prop}\label{idd}
Consider partitions $A, B\in T^d_w(\delta)$ with $B>A$.
The following diagram commutes:
\begin{equation*}
    \begin{tikzcd}
    K_0\left(\mathbb{S}_A(\delta)\right) \arrow[r, "i_*"]\arrow[d, "\Delta_{AB}"]& K_0\left(\mathbb{M}_A(\delta)\right)\arrow[d, "\Delta_{AB}"]\\
    K_0\left(\mathbb{S}_B(\delta)\right)\arrow[r, "i_*"]& K_0\left(\mathbb{M}_B(\delta)\right).
    \end{tikzcd}
\end{equation*}
\end{prop}

\begin{proof}
The maps $i_*$ and $p_{\lambda}^*$ commute. The functors $\beta_m$ are inverses of functors $q_{e,f}^*$, and the compatibility between $i_*$ and $q_{e,f}^*$ is treated in \cite[Proof of Proposition 3.6]{P1}.  
\end{proof}

\subsubsection{}\label{swapp} 
Assume that $Q$ is connected.
Consider a partition $A$ of $(d,w)$ with terms $(e,v)$ and $(f,u)$. Recall that $N_{e,f}:=N^{\lambda_{e,f}<0}$, $\rho_{d,e}:=\rho^{\lambda_{e,f}<0}$, and $L_{e,f}$ from Subsection \ref{gammadelta}. 

Let
\[\chi_A:=rN_{e,f}-\rho_{e,f}-\delta.\]
Let $A'$ be the partition of $(d,w)$ with parts $(f,u')$ and $(e,v')$ corresponding to the weight
\[
    \chi_{A'}:=w_{e,f}\chi_A+L_{f,e}.
\]
Using Proposition \ref{Wcomputation}, we see that
\begin{equation}\label{chichi}\chi_{A'}=(1-r)N_{f,e}-\rho_{f,e}-\delta,\end{equation}   so $\chi_{A'}$ corresponds to a partition $A'$ with terms $(f,u'), (e,v')$.
Denote the transformation \begin{equation}\label{transff}
    (e,v), (f,u)\mapsto (f,u'), (e,v').
\end{equation}
by $A\mapsto A'$. 
If $A\in T^d_w(\delta)$, then $A'\in T^d_w(\delta)$ by the computation in \eqref{chichi}. 
Define
\begin{align*}
    \text{sw}: K_0\left(\mathbb{S}(e; \delta_e)_v\right)\boxtimes
    K_0\left(\mathbb{S}(f;\delta_f)_u\right) &\to K_0\left(\mathbb{S}(f; \delta'_f)_{u'}\right)\boxtimes K_0\left(\mathbb{S}(e; \delta'_e)_{v'}\right)\\
    y&\mapsto (-1)^{\chi(e,f)}w_{e,f}(y)
    q^{L_{f,e}}.
    \end{align*}
Then $\text{sw}\circ\text{sw}=\text{id}$.

More generally, 
consider a partition $A=(d_i, w_i)_{i=1}^k$ of $(d,w)$. Using the transformation \eqref{transff} for all pairs $(d_i, d_{i+1})$ or $1\leq i\leq k-1$, we obtain an action of $\mathfrak{S}_k$ on partitions $A$ of $(d,w)$ of cardinal $k$ which we denote by $A\mapsto \sigma(A)$.
For $A$ and $B$ conjugate under $\mathfrak{S}_{|A|}$, we write $A\sim B$. 
If $A\in T^d_w(\delta)$, then all its conjugates are in $T^d_w(\delta)$. 
For $\sigma\in \mathfrak{S}_k$, there is a corresponding map
\[\text{sw}_\sigma: K_0(\mathbb{S}_A(\delta))\to K_0\left(\mathbb{S}_{\sigma(A)}(\delta)\right).\]

If $Q=Q^o$, then $T^d_w$ is zero for $d>1$ and has one element $(1,w)$ for $d=1$. The swap morphisms are all the identity. 

If $Q$ is not connected, extend the definition of $\text{sw}_{\sigma}$ using the decompositions from Subsection \ref{swe}.


\subsubsection{}\label{subsec541}

We now state the main result of this Subsection. First, consider pairs $(b,t)$, $(c,s)$, $(e,v)$, $(f,u)$, and $(d,w)$ in $\mathbb{N}^I\times\mathbb{Z}$ such that 
\[(e,v)+(f,u)=(b,t)+(c,s)=(d,w).\]
We denote by $A$ the two term partition $(e,v)$, $(f,u)$, and by $B$ the two term partition $(b,t)$, $(c,s)$. Assume that $A$ and $B$ are in $T^d_w(\delta)$. 
Let $\textbf{S}$ be the set of partitions $C$ of $T^d_w(\delta)$ with terms $(a_i, \alpha_i)$ for $1\leq i\leq 4$, some of them possibly zero, such that
\begin{align}\label{alignrelations}
    \notag (a_1, \alpha_1)+(a_2, \alpha_2)&=(e,v),\\
    \notag (a_3, \alpha_3)+(a_4, \alpha_4)&=(f,u),\\
    (a_1, \alpha_1)+(a_3, \alpha'_3)&=(b,t),\\
    (a_2, \alpha'_2)+(a_4, \alpha_4)&=(c,s),
\end{align}
where $\alpha'_2, \alpha'_3$ are defined via transformation \eqref{transff}. The weight $\delta\in M^{\mathfrak{S}_d}_{\mathbb{R}}$ induces corresponding weights for all the other pairs involved.
For such $C\in \textbf{S}$, consider the swap morphism for the transposition $(23)\in \mathfrak{S}_4$ from Subsection \ref{swapp}. Define
\[ \widetilde{m\boxtimes m}:=\left(m\boxtimes m\right)
    \text{sw}_{(23)}: K_0(\mathbb{S}_C(\delta))\to K_0(\mathbb{S}_B(\delta)).\]
    Note that the image of $K_0(\mathbb{S}_C(\delta))$ under $\text{sw}_{23}$ is in $K_0(\mathbb{S}_{C'}(\delta))$, where $C'$ is the partition with terms $(a_1, \alpha_1)$, $(a_3, \alpha'_3)$, $(a_2, \alpha'_2)$, and $(a_4, \alpha_4)$, and thus the image of $K_0(\mathbb{S}_{C'}(\delta))$ under $m\boxtimes m$ is in $K_0(\mathbb{S}_B(\delta))$ from the relation \eqref{alignrelations}.
Recall the definition of $\text{KHA}(Q,W)'$ and $\text{KHA}(Q,W)'_{d,w}$ from Subsection \ref{pbwkhas} and of $K_0(\mathbb{S}_A(\delta))'$ from \eqref{imim}. 

\begin{thm}\label{bial}
The following diagram commutes:
\begin{equation*}
    \begin{tikzcd}
    K_0(\mathbb{S}_A(\delta))'
    \arrow[d, "\Delta_{AC}"]\arrow[r, "m"]& 
    K_0\left(\mathbb{S}(d; \delta)_w\right)'\arrow[d, "\Delta_B"]\\
    \bigoplus_{C\in \textbf{S}}
    K_0(\mathbb{S}_C(\delta))'
    \arrow[r, "\widetilde{m\boxtimes m}"]& K_0(\mathbb{S}_B(\delta))'.
    \end{tikzcd}
\end{equation*}
\end{thm}

\begin{proof}[Proof of Theorem \ref{bial}]
We may assume that $Q$ is connected.

If $Q=Q^o$, then $d=1$, and the only possibility for $A$ and $B$ are the identity partition $(1,w)$. Thus the set $\textbf{S}=\{(1,w)\}$, and the statement is immediate.

Assume that $Q$ is not $Q^o$.
Recall the algebra morphism
\[i_*: \text{KHA}(Q,W)\to \text{KHA}(Q,0).\]
By Proposition \ref{idd}, $i_*$ commutes with $\Delta$. 
It suffices to check the statement in the zero potential case. The statement now follows from \cite[Theorem 5.2]{P3}; the statement in loc. cit. for the stack $R(d)/PGL(d)$ implies Theorem \ref{bial} for $w=0$ and $\delta=0$.
The proof of \cite[Theorem 5.2]{P3} is based on explicit shuffle formula for $m$ and $\Delta$, and the same computation works for all weights $w\in\mathbb{Z}$ and $\delta\in M^{\mathfrak{S}_d}_{\mathbb{R}}$. 
\end{proof}

\begin{prop}\label{dcomm}
Let $x_{e,v}\in \text{KHA}(Q,W)'_{e,v}$ and $x_{f,u}\in \text{KHA}(Q,W)'_{f,u}$. Then
\[x_{e,v}\cdot x_{f,u}=\left(x_{f,u}q^{\gamma(f,e)}\right)\cdot\left(x_{e,v}q^{-\delta(f,e)}\right).\]
\end{prop}

\begin{proof}
It suffices to check the statement in the zero potential case, when it follows from a direct computation using shuffle formulas, see \cite[Proposition 5.5]{P3}.
\end{proof}

\subsection{Spaces of generators}\label{pbwth}
One can show using induction and Theorem \ref{them1} that there is a Künneth-type isomorphism
\begin{equation}\label{kunn}
    \bigotimes_{i=1}^k K_0\left(\mathbb{S}(d_i; \delta_i)_{w_i}\right)'_{\mathbb{Q}}\cong
 K_0(\mathbb{S}_A(\delta))'_{\mathbb{Q}}
\end{equation}
for $A=(d_i, w_i)_{i=1}^k$ in $S^d_w(\delta)$.

We define inductively on $(d,w)$ a (split) subspace \[\ell_{d,w}: P(d; \delta)_w\hookrightarrow K_0\left(\mathbb{S}(d; \delta)_w\right)'_{\mathbb{Q}}\] with a surjection
\[\pi_{d,w}: K_0\left(\mathbb{S}(d; \delta)_w\right)'_{\mathbb{Q}}\twoheadrightarrow P(d; \delta)_w\] such that $\pi_{d,w}\ell_{d,w}=\text{id}$.
For any $A\in T^d_w(\delta)$, denote by $|A|$ the number of parts of $A$. For $d\in\mathbb{N}^I$ with $\sum_{i\in I}d^i=1$, let  \[P(d; \delta)_w:=K_0\left(\mathbb{S}(d; \delta)_w\right)'_{\mathbb{Q}}.\] 
Let $A=(d_i,w_i)_{i=1}^k\in T^d_w(\delta)$ with $A>(d,w)$, or equivalently with $k\geq 2$. Let $P_A(\delta):=\otimes_{i=1}^k P(d_i; \delta_i)_{w_i}$ and let $\pi_A$ be the natural projection:
\[\pi_A:=\otimes_{i=1}^k \pi_{d_i, w_i}: K_0(\mathbb{S}_A(\delta))'_{\mathbb{Q}}\twoheadrightarrow P_A(\delta).\]
Let $K_A$ be the kernel of the map
\[\left(\bigoplus_{\sigma\in\mathfrak{S}_{|A|}}
\pi_{\sigma(A)}\right) \left(\bigoplus_{\sigma\in\mathfrak{S}_{|A|}}
\text{sw}_\sigma\right)\Delta_{A}:
K_0(\mathbb{S}(d; \delta)_w)'_{\mathbb{Q}}\to \bigoplus_{\sigma\in\mathfrak{S}_{|A|}}P_{\sigma(A)}(\delta).\]
Define
\[P(d;\delta)_w:=\bigcap_{A>(d,w)}
 K_A\hookrightarrow K_0\left(\mathbb{S}(d; \delta)_w\right)_{\mathbb{Q}}.\]
 Let $O\subset T^d_w(\delta)\setminus \{(d,w)\}$ be a set which contains exactly one set in any $\mathfrak{S}$-orbit.
The following is proved as in \cite[Theorem 5.13]{P3}, the loc. cit. treats the case of stacks $R(d)/PGL(d)$ and applies directly to the case when $w=0$ and $\delta=0$. 
\begin{prop}\label{513}
There is a decomposition 
\[P(d;\delta)_w\oplus \bigoplus_{A\in O} \left(\bigoplus_{\sigma\in\mathfrak{S}_{|A|}}P_{\sigma(A)}(\delta)\right)^{\mathfrak{S}_{|A|}}\xrightarrow{\sim} K_0\left(\mathbb{S}(d; \delta)_w\right)'_{\mathbb{Q}}.\]
In particular, there is a natural projection map
\[\pi_{d,w}: K_0\left(\mathbb{S}(d; \delta)_w\right)'_{\mathbb{Q}}\twoheadrightarrow P(d;\delta)_w\] such that $\pi_{d,w}\ell_{d,w}=\text{id}$.
\end{prop}


\begin{proof}[Proof of Theorem \ref{them2}]
The space $K_0(\text{MF}(\X(d), W)_d)$ is generated by $P_A(\delta)$ for $A\in U^d_w(\delta)$ by Theorem \ref{them1} and Proposition \ref{513}. There are no relations between the generators $y_A$ of $P_A(\delta)$ for $A=(d_i, w_i)\in S^d_w(\delta)$ by Theorem \ref{them1}. The relations \eqref{relthem2} are satisfied by Proposition \ref{dcomm}. There are no other relations between generators of $P_A(\delta)$ for $A\in \widetilde{T^d_w(\delta)}$ by Proposition \ref{513}. 
\end{proof}


\end{document}